\documentclass[reqno]{amsart}

\usepackage{tikz-cd}

\usepackage{fullpage,dsfont}

\usepackage{hyperref} % this should come before amsrefs

\usepackage{parskip}
\makeatletter % need this to avoid the conflict between amsthm and parskip
\def\thm@space@setup{%
 \thm@preskip=\parskip \thm@postskip=0pt
}
\def\th@remark{%
  \thm@headfont{\itshape}%
  \normalfont % body font
  \thm@preskip\parskip \thm@postskip=0pt
}
\makeatother

\usepackage[nobysame,alphabetic,initials,msc-links]{amsrefs}

\usepackage[final]{pdfpages}

\usepackage{multirow}

\usepackage{stmaryrd}

\usepackage{array}

\DefineSimpleKey{bib}{how}
\DefineSimpleKey{bib}{mrclass}
\DefineSimpleKey{bib}{mrnumber}
\DefineSimpleKey{bib}{fjournal}
\DefineSimpleKey{bib}{mrreviewer}

\renewcommand{\PrintDOI}[1]{%
  \href{http://dx.doi.org/#1}{{\tt DOI:#1}}%
}
\renewcommand{\eprint}[1]{#1}
\BibSpec{book}{%
    +{}  {\PrintPrimary}                {transition}
    +{.} { \PrintDate}                  {date}
    +{.} { \textit}                     {title}
    +{.} { }                            {part}
    +{:} { \textit}                     {subtitle}
    +{,} { \PrintEdition}               {edition}
    +{}  { \PrintEditorsB}              {editor}
    +{,} { \PrintTranslatorsC}          {translator}
    +{,} { \PrintContributions}         {contribution}
    +{,} { }                            {series}
    +{,} { \voltext}                    {volume}
    +{,} { }                            {publisher}
    +{,} { }                            {organization}
    +{,} { }                            {address}
    +{,} { }                            {status}
    +{,} { \PrintDOI}                   {doi}
    +{,} { \PrintISBNs}                 {isbn}
    +{}  { \parenthesize}               {language}
    +{}  { \PrintTranslation}           {translation}
    +{;} { \PrintReprint}               {reprint}
    +{.} { }                            {note}
    +{.} {}                             {transition}
    +{}  {\SentenceSpace \PrintReviews} {review}
}
\BibSpec{article}{%
    +{}  {\PrintAuthors}                {author}
    +{,} { \textit}                     {title}
    +{.} { }                            {part}
    +{:} { \textit}                     {subtitle}
    +{,} { \PrintContributions}         {contribution}
    +{.} { \PrintPartials}              {partial}
    +{,} { }                            {journal}
    +{}  { \textbf}                     {volume}
    +{}  { \PrintDatePV}                {date}
    +{,} { \issuetext}                  {number}
    +{,} { \eprintpages}                {pages}
    +{,} { }                            {status}
    +{,} { \PrintDOI}                   {doi}
    +{,} { \eprint}        {eprint}
    +{}  { \parenthesize}               {language}
    +{}  { \PrintTranslation}           {translation}
    +{;} { \PrintReprint}               {reprint}
    +{.} { }                            {note}
    +{.} {}                             {transition}
    +{}  {\SentenceSpace \PrintReviews} {review}
}
\BibSpec{collection.article}{%
    +{}  {\PrintAuthors}                {author}
    +{,} { \textit}                     {title}
    +{.} { }                            {part}
    +{:} { \textit}                     {subtitle}
    +{,} { \PrintContributions}         {contribution}
    +{,} { \PrintConference}            {conference}
    +{}  {\PrintBook}                   {book}
    +{,} { }                            {booktitle}
    +{,} { \PrintDateB}                 {date}
    +{,} { pp.~}                        {pages}
    +{,} { }                            {publisher}
    +{,} { }                            {organization}
    +{,} { }                            {address}
    +{,} { }                            {status}
    +{,} { \PrintDOI}                   {doi}
    +{,} { \eprint}        {eprint}
    +{}  { \parenthesize}               {language}
    +{}  { \PrintTranslation}           {translation}
    +{;} { \PrintReprint}               {reprint}
    +{.} { }                            {note}
    +{.} {}                             {transition}
    +{}  {\SentenceSpace \PrintReviews} {review}
}
\BibSpec{misc}{%
  +{}{\PrintAuthors}  {author}
  +{,}{ \textit}      {title}
  +{.}{ }             {how}
  +{}{ \parenthesize} {date}
  +{,} { available at \eprint}        {eprint}
  +{,}{ available at \url}{url}
  +{,}{ }             {note}
  +{.}{}              {transition}
}
\usepackage{amssymb, amsfonts, amsxtra, amsmath}
\usepackage{mathrsfs}
\usepackage{mathdots}
\usepackage{wasysym}
\usepackage[all]{xy}
\usepackage{bbm}
\usepackage{calc}
\usepackage{accents}

\usepackage[bbgreekl]{mathbbol}			% for \mathbb greek letters (e.g. $\mathbbl{\Lambda}$)

\usepackage{stackengine}

\numberwithin{equation}{section}

\DeclareSymbolFontAlphabet{\mathbb}{AMSb}	% these two magic spells take care of the problem
\DeclareSymbolFontAlphabet{\mathbbl}{bbold}	% of redefining \mathbb symbols by mathbbol package

\newtheorem{Theorem}{Theorem}[section]
\newtheorem*{Theorem*}{Theorem}
\newtheorem{Def}[Theorem]{Definition}
\newtheorem*{Def*}{Def}
\newtheorem{Lem}[Theorem]{Lemma}
\newtheorem{Prop}[Theorem]{Proposition}
\newtheorem{Cor}[Theorem]{Corollary}

\newtheorem{Rem}[Theorem]{Remark}

\newtheorem{Exa}[Theorem]{Example}

\newcommand\bp{\begin{proof}}
\newcommand\ep{\end{proof}}

\mathchardef\mhyph="2D

\DeclareMathOperator{\id}{\mathrm{id}}

\DeclareMathOperator{\Rep}{\mathrm{Rep}}

\DeclareMathOperator{\Irr}{\mathrm{Irr}}

\newcommand{\op}{\mathrm{op}}

\newcommand{\msN}{\mathscr{N}}

\newcommand{\mcH}{\mathcal{H}}

\newcommand{\C}{\mathbb{C}}
\newcommand{\G}{\mathbb{G}}

\newcommand{\Corr}{\mathrm{Corr}}

\newcommand{\ovot}{\bar{\otimes}}

\allowdisplaybreaks

\begin{document}
\author{Joeri De Ro}
\address{Department of Mathematics and Data Science, Vrije Universiteit Brussel, Brussels, Belgium}
\email{joeri.ludo.de.ro@vub.be}
\title{Morita theory for dynamical von Neumann algebras}

\begin{abstract}
Given a locally compact quantum group $\G$ and two $\G$-$W^*$-algebras $\alpha: A\curvearrowleft \G$ and $\beta: B\curvearrowleft \G$, we study the notion of equivariant $W^*$-Morita equivalence $(A, \alpha)\sim_\G (B, \beta)$, which is an equivariant version of Rieffel's notion of $W^*$-Morita equivalence. We prove that important dynamical properties of $\G$-$W^*$-algebras, such as (inner) amenability, are preserved under equivariant Morita equivalence. For a coideal von Neumann algebra $L^\infty(\mathbb{K}\backslash \G)\subseteq L^\infty(\G)$ with dual coideal von Neumann algebra $L^\infty(\check{\mathbb{K}})\subseteq L^\infty(\check{\G})$, we use a natural $\check{\G}$-$W^*$-Morita equivalence $L^\infty(\mathbb{K}\backslash \G)\rtimes_\Delta \G \sim_{\check{\G}} L^\infty(\check{\mathbb{K}})$ to relate dynamical properties of $L^\infty(\mathbb{K}\backslash \G)$ with dynamical properties of $L^\infty(\check{\mathbb{K}})$. We use this to refine some recent results established by Anderson-Sackaney and Khosravi. This refinement allows us to answer a question of Kalantar, Kasprzak, Skalski and Vergnioux, namely that for $\mathbb{H}$ a closed quantum subgroup of the compact quantum group $\G$, coamenability of $\mathbb{H}\backslash \G$ and relative amenability of $\ell^\infty(\check{\mathbb{H}})$ in $\ell^\infty(\check{\G})$ are equivalent. Moreover, if $\G$ is compact, we study the relation between $\G$-$W^*$-Morita equivalence of $(A, \alpha)$ and $(B, \beta)$ and $\G$-$C^*$-Morita equivalence of the associated $\G$-$C^*$-algebras $(\mathcal{R}(A), \alpha)$ and $(\mathcal{R}(B), \beta)$ of regular elements.
\end{abstract}
\maketitle

\section{Introduction}

Given a von Neumann algebra $A$, let us denote the $W^*$-category of normal, unital $*$-representations of $A$ on Hilbert spaces by $\Rep(A)$. In the seminal paper \cite{Rie74}, Rieffel called two von Neumann algebras $A$ and $B$ \emph{$W^*$-Morita equivalent} if $\Rep(A)$ and $\Rep(B)$ are equivalent as $W^*$-categories. In fact, we have the following result:
\begin{Theorem}\label{Morita characterizations} The following statements are equivalent:
    \begin{enumerate}\setlength\itemsep{-0.5em}
    \item\label{category} $\Rep(A)$ and $\Rep(B)$ are equivalent as $W^*$-categories.
    \item\label{correspondence} There exists a Hilbert space $\mathcal{H}$, a faithful unital normal $*$-representation $\pi: A\to B(\mathcal{H})$ and a faithful unital normal anti-$*$-representation $\rho: B \to B(\mathcal{H})$ such that $\pi(A) = \rho(B)'$.
    \item\label{linking} There exists a linking von Neumann algebra between $A$ and $B$, i.e.\ a von Neumann algebra $E$ together with full projections $p, 1-p\in E$ such that $pEp\cong A$ and $(1-p)E(1-p)\cong B$.
    \item\label{corner} There exists a Hilbert space $\mathcal{H}$ and a full projection $p\in B(\mathcal{H})\ovot B$ such that $A \cong p(B(\mathcal{H})\ovot B)p.$
\end{enumerate}
\end{Theorem}

From the above equivalences, it is easily deduced that $W^*$-Morita equivalence preserves important von Neumann algebraic properties, such as injectivity (= semidiscreteness), factoriality, etc. Therefore, the notion of $W^*$-Morita equivalence provides a useful tool in von Neumann algebra theory for comparing von Neumann algebras. We note that the equivalences between the above statements are well-known, see e.g.\ \cites{Rie74, BDH88, Bro03,DC09}.

The main goal of this paper is to study the notion of \emph{equivariant} Morita equivalence. More concretely, we consider a locally compact quantum group $\G$ and we upgrade the von Neumann algebras $A$ and $B$ with actions $\alpha: A\curvearrowleft \G$ and $\beta: B\curvearrowleft \G$. We refer to these as $\G$-$W^*$-algebras. As a starting point, we consider the following definition, which is an equivariant version of \eqref{correspondence} above:

\begin{Def}[\cite{DCDR24a}*{Definition 5.8}] \label{MoritaClassical} Let $\G$ be a locally compact quantum group.
    The $\G$-$W^*$-algebras $(A, \alpha)$ and $(B, \beta)$ are called $\G$-$W^*$-Morita equivalent, and we write $(A, \alpha)\sim_\G (B, \beta)$, if there exists a Hilbert space $\mathcal{H}$, endowed with the following data:
    \begin{itemize}\setlength\itemsep{-0.5em}
        \item A normal, unital faithful $*$-representation $\pi: A \to B(\mathcal{H})$,
        \item a normal, unital faithful anti-$*$-representation $\rho: B \to B(\mathcal{H})$,
        \item a unitary $\G$-representation $U\in B(\mathcal{H})\ovot L^\infty(\G)$,
        \end{itemize}
        such that $\pi(A)'= \rho(B)$ and
        $$(\pi\otimes \id)\alpha(a) = U(\pi(a)\otimes 1)U^*, \quad (\rho \otimes R)\beta(b) = U^*(\rho(b)\otimes 1)U, \quad a\in A, \quad b\in B.$$
\end{Def}

It should be mentioned that the notion of equivariant Morita equivalence has already been considered earlier throughout the quantum group literature, see e.g.\ \cites{Vae05, DC12} where explicit examples of equivariant Morita equivalences are considered in the $C^*$-context. 

As a first main task, we will give a complete generalization of Theorem \ref{Morita characterizations} to the \emph{equivariant setting}. 
In this regard, we note that an equivariant version of the equivalence $(1) \iff (2)$ has already been established in \cite{DR24b}*{Theorem 3.10}. The equivalence \eqref{correspondence} $\iff$ \eqref{linking} takes the following obvious form in the equivariant setting:

\textbf{Theorem \ref{char1}.} \textit{Let $\G$ be a locally compact quantum group and $(A, \alpha)$ and $(B, \beta)$ two $\G$-$W^*$-algebras. Then $(A, \alpha)$ and $(B, \beta)$ are $\G$-$W^*$-Morita equivalent if and only if there exists a $\G$-$W^*$-algebra $(E, \gamma)$, a projection $p\in E^\gamma$ such that $p$ and $1-p$ are full projections of $E$ and $\G$-equivariant isomorphisms $pEp \cong A$ and $(1-p)E(1-p) \cong B$.}

The correct version of \eqref{corner} in the equivariant setting is more subtle, as additional equivariant data appears:

\textbf{Theorem \ref{char2}.}\textit{ Let $\G$ be a locally compact quantum group and let $(A, \alpha)$ and $(B, \beta)$ be $\G$-$W^*$-algebras. The following are equivalent:}
      \begin{enumerate}\setlength\itemsep{-0.5em}
        \item \textit{$(A, \alpha)\sim_\G (B, \beta)$.}
        \item \textit{There exists a Hilbert space $\mathcal{K}$, a full projection $p\in B(\mathcal{K})\ovot B$, $P\in B(\mathcal{K})\ovot B \ovot L^\infty(\G)$ and a $*$-isomorphism $\phi: p(B(\mathcal{K})\ovot B)p\to A$ such that the following conditions are satisfied:}
            \begin{enumerate}\setlength\itemsep{-0.5em}
        \item \textit{$P^*P=(\id \otimes \beta)(p)$.}
        \item \textit{$PP^* = p\otimes 1$.}
        \item \textit{$(\id \otimes \id \otimes \Delta)(P) = (P\otimes 1)(\id \otimes \beta \otimes \id)(P).$}
        \item \textit{$\alpha(\phi(z)) = (\phi\otimes \id)(P(\id \otimes \beta)(z)P^*)$ for all $z\in p(B(\mathcal{K})\ovot B)p.$}
    \end{enumerate}
    \end{enumerate}

In other words, one action is obtained by amplifying the other action and then twisting it with some kind of (non-unitary) cocycle.

We give examples of important equivariant Morita equivalences that appear in practice. For example:

\textbf{Proposition \ref{coideal Morita}}\textit{ Let $\G$ be a locally compact quantum group and let $L^\infty(\mathbb{K}\backslash \G)$ be a coideal von Neumann subalgebra of $L^\infty(\G)$ with associated dual coideal von Neumann subalgebra $L^\infty(\check{\mathbb{K}})\subseteq L^\infty(\check{\mathbb{G}})$. There is a $\check{\G}$-$W^*$-Morita equivalence $L^\infty(\check{\mathbb{K}}) \sim_{\check{\G}} L^\infty(\mathbb{K}\backslash \G)\rtimes_\Delta \G.$}

In the setting of quantum groups, the notions of (strong) $\G$-amenability and (strong) inner $\G$-amenability for $\G$-$W^*$-algebras were introduced and studied in detail in the papers \cites{Moa18, DCDR24a, DCDR24b, DR24a} and the notion of $\G$-injectivity was studied in detail in the papers \cites{CN16, Cr17a, DR24a}. Using Theorem \ref{char2}, we can then prove that $\G$-$W^*$-Morita equivalence preserves these important dynamical properties:

\textbf{Theorem \ref{MoritaApplication}} \textit{Let $(A, \alpha)$ and $(B, \beta)$ be $\G$-$W^*$-Morita equivalent. The following statements hold:}
    \begin{enumerate}\setlength\itemsep{-0.5em}
        \item \textit{$(A, \alpha)$ is (strongly) $\G$-amenable if and only if $(B, \beta)$ is (strongly) $\G$-amenable.}
\item \textit{$(A, \alpha)$ is (strongly) inner $\G$-amenable if and only if $(B, \beta)$ is (strongly) inner $\G$-amenable.}
\item \textit{$(A, \alpha)$ is $\G$-injective if and only if $(B, \beta)$ is $\G$-injective.}
\end{enumerate}

We use Theorem \ref{MoritaApplication} to solve some open questions in the literature, illustrating the use of equivariant Morita equivalence in the theory of locally compact quantum groups. For example, applying Theorem \ref{MoritaApplication} to the $\check{\G}$-$W^*$-Morita equivalence in Proposition \ref{coideal Morita}, we can refine the important result \cite{AK24}*{Theorem 5.9}. For the definition of compact quasi-subgroup, we refer the reader to \cite{AK24}*{Section 2.4}.

\textbf{Theorem \ref{main}.} \textit{Let $\G$ be a compact quantum group and $L^\infty(\mathbb{K}\backslash \G)$ a (right) coideal von Neumann algebra. The following statements are equivalent:}
\begin{enumerate}\setlength\itemsep{-0.5em}
    \item \textit{$\ell^\infty(\check{\mathbb{K}})$ is $\check{\G}$-amenable.}
    \item \textit{$\ell^\infty(\check{\mathbb{K}})$ is $\check{\G}$-injective.}
    \item \textit{$L^\infty(\mathbb{K}\backslash \G)$ is inner $\G$-amenable.}
    \item \textit{$L^\infty(\mathbb{K}\backslash \G)$ is $\G$-injective.}
    \item \textit{$\mathbb{K}$ is a compact quasi-subgroup of $\G$ and $\ell^\infty(\check{\mathbb{K}})$ is relatively amenable in $\ell^\infty(\check{\G})$, i.e.\ there exists a $\check{\G}$-equivariant ucp map $\ell^\infty(\check{\G})\to \ell^\infty(\check{\mathbb{K}}).$}
\end{enumerate}

In particular, applying this result to a closed quantum subgroup $\mathbb{K}= \mathbb{H}$ of the compact quantum group $\G$ and using the results of \cite{AK24}, we find:

\textbf{Corollary \ref{ezcor}.}  Let $\mathbb{H}$ be a closed quantum subgroup of the compact quantum group $\G$. Then $\mathbb{H}\backslash \G$ is co-amenable if and only if $\ell^\infty(\check{\mathbb{H}})$ is relative amenable in $\ell^\infty(\check{\G})$.

This gives an affirmative answer to the open question \cite{KKSV22}*{Question 8.1}.

Finally, if $\G$ is a compact quantum group and $(A, \alpha)$ is a $\G$-$W^*$-algebra, we write $\mathcal{R}(A)$ for its space of \emph{regular elements}. Then $(\mathcal{R}(A), \alpha)$ is a $\sigma$-weakly dense $\G$-$C^*$-subalgebra of $A$. We use the theory of self-dual Hilbert modules \cite{Pa73} and the results of \cite{DC12} in the $C^*$-setting to prove the following result:

\textbf{Theorem \ref{ergodic}.}\textit{ Let $\G$ be a compact quantum group and let $(A, \alpha)$ and $(B, \beta)$ be two $\G$-$W^*$-algebras. Consider the following statements:}
    \begin{enumerate}\setlength\itemsep{-0.5em}
        \item \textit{$(A, \alpha)$ and $(B, \beta)$ are $\G$-$W^*$-Morita equivalent.}
        \item \textit{$(\mathcal{R}(A), \alpha)$ and $(\mathcal{R}(B), \beta)$ are $\G$-$C^*$-Morita equivalent.}
    \end{enumerate}
\textit{Then $(2)\implies (1)$ and $(1) \iff (2)$ if both $(A, \alpha)$ and $(B, \beta)$ are ergodic.}

It then immediately follows that the main result of \cite{DC12} is also valid on the von Neumann algebra level:

\textbf{Corollary \ref{classification}.} \textit{The classification of Podleś spheres up to $SU_q(2)$-$C^*$-Morita equivalence, as established in \cite{DC12}*{Theorem 0.1}, is also valid on the von Neumann algebra level. }
    
\section{Preliminaries}

All vector spaces and algebras in this paper are defined over the complex numbers. Inner products of Hilbert spaces are assumed to be anti-linear in the first factor. If $V,W$ are normed vector spaces, we write $B(V,W)$ for the vector space of bounded linear operators $V\to W$, naturally normed through the operator norm and we write $B(V):= B(V,V)$. If $(V, \tau)$ is a topological vector space and $X$ is a subset of $V$, then we write $[X]^\tau$ for the closure of the linear span of $X$. If $V$ is a normed space, we simply write $[X]$ for the norm-closure of the linear span of $X$. 
Let $C$ be a $C^*$-algebra. By a projection $p$ in $C$, we always mean an element $p\in C$ such that $p= p^2 = p^*$, and if $C$ is unital, we write $p^\perp := 1-p$. We assume that the inner product of a right Hilbert $C$-module $X$ is anti-linear in the first factor and the inner product of a left Hilbert $C$-module is anti-linear in the second factor. The $C^*$-algebra of $C$-adjointable operators on the right Hilbert-$C$-module $X$ will be denoted by $\mathscr{L}_C(X)$. 

\subsection{Von Neumann algebras} We shall use the terminologies von Neumann algebra and $W^*$-algebra as synonyms. Given a von Neumann algebra $A$, we write $A_*$ for its space of normal functionals. A projection $p\in A$ is called \emph{full} if $[ApA]^{\sigma\text{-weak}}= A$. We quickly go over some of the notations in Tomita-Takesaki theory \cite{Tak03} that we will be using.

The standard Hilbert space of $A$ will be denoted by $L^2(A)$, together with its natural standard representation 
$\pi_A: A \to B(L^2(A)).$
This Hilbert space is endowed with a natural anti-unitary $J_A: L^2(A)\to L^2(A)$ satisfying $J_A= J_A^{-1}= J_A^*$. We denote the standard anti-$*$-representation by 
$\rho_A: A \to B(L^2(A)): a \mapsto J_A \pi_A(a)^* J_A.$
We have $\rho_A(A)' = \pi_A(A).$

If $\varphi$ is a normal, semi-finite, faithful (= nsf) weight on $A$, we write
$$\mathscr{M}_\varphi^+ :=\{a\in A^+: \varphi(a) <\infty\}, \quad \mathscr{N}_\varphi:= \{a\in A: \varphi(a^*a) < \infty\}.$$
With respect to the nsf weight $\varphi$, we can make a GNS construction $\pi_\varphi: A\to B(L^2(A, \varphi))$ where $\pi_\varphi$ is a unital, normal $*$-homomorphism. We will make use of the associated GNS-map $\Lambda_\varphi: \mathscr{N}_\varphi\to L^2(A, \varphi)$. The Hilbert space $L^2(A, \varphi)$ is naturally endowed with the anti-unitary
$J_\varphi: L^2(A, \varphi)\to L^2(A, \varphi)$ satisfying $J_\varphi= J_\varphi^* = J_\varphi^{-1}.$ 
There is a (essentially by construction of $L^2(A)$) canonical unitary identification $L^2(A)\cong L^2(A, \varphi)$ compatible with the standard representations and the canonical anti-unitaries. 

\subsection{Correspondences and W*-Morita equivalence} \label{corr}

Let $A,B,C$ be von Neumann algebras. We introduce some terminology. 

\begin{Def}\cite{Rie74}*{Definition 5.1}
    A right Hilbert $W^*$-$A$-$B$-bimodule consists of a right Hilbert $B$-module $X$, together with a unital $*$-homomorphism $\pi: A \to \mathscr{L}_B(X)$ such that the maps $A\to B: a \mapsto \langle x, ay\rangle_B$ are normal for all $x,y \in X$.
\end{Def}

Similarly, we define the notion of a \emph{left} Hilbert $W^*$-$A$-$B$-bimodule $X$, where we work with a left Hilbert $A$-module $X$ and a unital anti-$*$-representation $B \to {}_A\mathscr{L}(X)$.

\begin{Def}\cite{Con80+}\label{correspondence_definition}
    An $A$-$B$-correspondence consists of a Hilbert space $\mathcal{H}$ together with a unital, normal $*$-representation $\pi: A \to B(\mathcal{H})$ and a unital, normal anti-$*$-representation $\rho: B \to B(\mathcal{H})$ such that $\pi(a)\rho(b) = \rho(b)\pi(a)$ for all $a\in A$ and $b\in B$. We write $\mathcal{H}= (\mathcal{H}, \pi, \rho)\in \Corr(A,B)$.
\end{Def}

There is a way to `compose' a (right) Hilbert $W^*$-$A$-$B$-bimodule $X$ and a $B$-$C$-correspondence $\mathcal{H}$ to obtain an $A$-$C$-correspondence $X \otimes_B \mathcal{H}$. Indeed, consider the balanced (algebraic) tensor product $X \odot_B\mathcal{H}$ and endow it with the semi-inner product uniquely determined by
$$\langle x\otimes_B \xi, y \otimes_B \eta\rangle = \langle \xi, \pi_\mathcal{H}(\langle x,y\rangle_B)\eta\rangle, \quad x,y \in X, \quad \xi, \eta \in \mathcal{H}.$$
By separation-completion of $X\odot_B\mathcal{H}$, we then obtain the Hilbert space $X\otimes_B\mathcal{H}$. It carries a natural $A$-$C$-correspondence structure through
$$\pi_\boxtimes(a)(x\otimes_B \xi):= ax\otimes_B \xi, \quad \rho_\boxtimes(c)(x\otimes_B \xi):=x \otimes_B \rho_\mathcal{H}(c)\xi, \quad a \in A, \quad c\in C, \quad x \in X, \quad \xi \in \mathcal{H}.$$

If $\mathcal{H}\in \Corr(A,B)$, then
$$X:= \mathscr{L}_B(L^2(B), \mathcal{H})= \{x\in B(L^2(B), \mathcal{H})\mid \forall b \in B: \rho_\mathcal{H}(b)x=x\rho_B(b)\}$$ becomes a (right) Hilbert $W^*$-$A$-$B$-bimodule through 
$$ax := \pi_\mathcal{H}(a)x, \quad xb:= x\pi_B(b), \quad \langle x,y\rangle_B := \pi_B^{-1}(x^*y), \quad x,y\in X, \quad a\in A, \quad b\in B.$$

We will mainly be interested in special kinds of Hilbert $W^*$-$A$-$B$-bimodules or $A$-$B$-correspondences, namely these that encode the important notion of a Morita equivalence between the von Neumann algebras $A$ and $B$.  To this end, we provide the following two definitions:

\begin{Def}\label{equivalence bimodule} \cite{Rie74}*{Definition 7.5}
    An equivalence $A$-$B$-bimodule $X$ consists of an $A$-$B$-bimodule $X$ endowed with an $A$-valued and a $B$-valued inner product that turns $X$ simultaneously into a left and a right Hilbert $W^*$-$A$-$B$-bimodule. Moreover, the following properties must be satisfied:
\begin{itemize}\setlength\itemsep{-0.5em}
    \item $\langle x,y\rangle_A z= x \langle y,z\rangle_B$ for all $x,y,z\in X.$
    \item $[\langle X,X\rangle_A]^{\sigma\text{-weak}}= A$ and $[\langle X,X\rangle_B]^{\sigma\text{-weak}}= B$.
\end{itemize}
\end{Def}

\begin{Def}\label{definitionn}
    A Morita $A$-$B$-correspondence consists of an $A$-$B$-correspondence $(\mathcal{H}, \pi, \rho)$ such that $\pi$ and $\rho$ are faithful and such that $\pi(A)'= \rho(B)$. 
\end{Def}

It is well-known how to move between these pictures of $W^*$-Morita equivalence $A \sim B$. Concretely, if $\mathcal{H}$ is a Morita correspondence between $A$ and $B$, then $X:= \mathscr{L}_B(L^2(B), \mathcal{H})$ becomes an equivalence $A$-$B$-bimodule through 
$$\langle x,y\rangle_A:= \pi_\mathcal{H}^{-1}(xy^*), \quad \langle x,y\rangle_B := \pi_B^{-1}(x^*y), \quad xb := x \pi_B(b), \quad ax := \pi_\mathcal{H}(a)x, \quad a\in A, \quad b\in B, \quad x,y \in X.$$

To see that the inner products satisfy the desired density conditions, we need the following well-known result:

\begin{Prop} Let $\rho: B \to B(\mathcal{H})$ be a unital, normal, anti-$*$-representation and put $X:= \mathscr{L}_B(L^2(B), \mathcal{H}).$
    The following statements are equivalent:
    \begin{enumerate}\setlength\itemsep{-0.5em}
        \item $\pi_B(B) = [X^*X]^{\sigma\text{-weak}}.$
        \item $\rho$ is injective.
    \end{enumerate}
    If moreover $(\mathcal{H}, \pi, \rho)$ is a Morita $A$-$B$-correspondence, then $[XX^*]^{\sigma\text{-weak}}= \pi(A).$
\end{Prop}
\begin{proof} $(1)\implies (2)$ Assume that $b\in B$ and $\rho(b)= 0$. Then for $y,z \in X$ we have
$\rho_B(b)y^*z = y^* \rho(b)z = 0$
and thus $\rho_B(b) = 0$, whence $b= 0$.

$(2)\implies (1)$ Assume that $\rho$ is faithful. Since $I:= [X^*X]^{\sigma\text{-weak}}$ is a $\sigma$-weakly closed ideal in $\pi_B(B)$, there is a central projection $p\in B$ (and thus $\rho_B(p) = \pi_B(p)$) such that $I = \pi_B(p) \pi_B(B)$. Therefore, if $x\in X$, we find that 
$x^*x = \pi_B(p) x^*x = \rho_B(p)x^*x = x^* \rho(p)x$
so that $x^*\rho(p^\perp)x= 0$. Therefore, $\rho(p^\perp)x=0$ by the $C^*$-identity. Since $X L^2(B)$ is norm-dense in $\mathcal{H}$ \cite{Tak03}*{IX.3 Lemma 3.3}, we deduce that $\rho(p^\perp) = 0$ so faithfulness yields $1= p$ and thus $I = \pi_B(B)$.

Assume next that $(\mathcal{H}, \pi, \rho)$ is a Morita $A$-$B$-correspondence. Put $J:= [XX^*]^{\sigma\text{-weak}}$ which is a $\sigma$-weakly closed ideal in $\pi(A)$. There exists a central projection $p\in A$ such that $\pi(p)\pi(A)= J$. Then clearly for every $x\in X$ we have 
$\pi(p^\perp)xx^* \pi(p^\perp) = \pi(p^\perp)xx^* = 0$
so that $\pi(p^\perp)x= 0$. Since $XL^2(B)$ is norm dense in $\mathcal{H}$, the faithfulness of $\pi$ implies that $p=1$.
\end{proof}

To go in the other direction, we can use the following useful result, which we will use several times in this paper. This result is surely known to experts, but we could not find a reference in the literature. We therefore give a self-contained proof.

\begin{Prop}\label{technical}
    Let $A,B,C$ be von Neumann algebras, let $X$ be an $A$-$B$-equivalence bimodule and let $\mathcal{K}$ be a Morita $B$-$C$-correspondence. Then  $X\otimes_B \mathcal{K}$ is a Morita $A$-$C$ correspondence.
\end{Prop}
\begin{proof}
    Given $x\in X$, let us define the bounded operator 
    $i_x: \mathcal{K}\to X\otimes_B \mathcal{K}$ by $i_x(\xi) = x\otimes_B \xi$. Its adjoint is then given by $i_x^*(y \otimes_B \xi) = \pi_\mathcal{K}(\langle x,y\rangle_B) \xi$ for $y\in X$ and $\xi \in \mathcal{K}$.

    Consider the subspace $$\mathcal{K}_0:= \operatorname{span}\{\pi_{\mathcal{K}}(\langle x,y\rangle_B)\xi: x,y \in X, \xi \in \mathcal{K}\}\subseteq \mathcal{K}.$$
    Since $[\langle X,X\rangle_B]^{\sigma\text{-strong}}=B$, we conclude that $\mathcal{K}_0$ is norm-dense in $\mathcal{K}$.
    Assume now that $T \in \pi_\boxtimes(A)'$ is a unitary. If $\sum_{j=1}^n \pi_\mathcal{K}(\langle x_j, y_j\rangle_B)\xi_j = 0$, then for $s,t \in X$ and $\eta \in \mathcal{K}$, we have
    \begin{align*}
\left \langle \sum_{j=1}^n i_{x_j}^* T(y_j\otimes_B \xi_j), \pi_\mathcal{K}(\langle s,t\rangle_B)\eta\right\rangle&= \sum_{j=1}^n \langle T(y_j\otimes_B \xi_j), x_j\otimes_B \pi_\mathcal{K}(\langle s,t\rangle_B)\eta\rangle\\
&= \sum_{j=1}^n \langle T(y_j\otimes_B \xi_j), \langle x_j, s\rangle_A t \otimes_B \eta\rangle\\
&=\sum_{j=1}^n \langle T(\langle s, x_j\rangle_A y_j\otimes_B \xi_j\rangle), t \otimes_B \eta\rangle\\
&=\sum_{j=1}^n \langle T(s \otimes_B \pi_\mathcal{K}(\langle x_j, y_j\rangle_B)\xi_j), t \otimes_B \eta\rangle = 0,\end{align*}
    so that $\sum_{j=1}^n i_{x_j}^* T(y_j\otimes_B \xi_j)= 0$. Therefore, there is a unique linear operator $S: \mathcal{K}_0\to \mathcal{K}$ such that 
   $$S(\pi_{\mathcal{K}}(\langle x,y\rangle_B)\xi) = i_x^* T(y\otimes_B \xi), \quad x,y \in X, \quad \xi \in \mathcal{K}.$$
For $x,y,s,t \in X$ and $\xi, \eta \in \mathcal{K}$, we have
\begin{align*}
    \langle \xi, \pi_{\mathcal{K}}(\langle x,y\rangle_B) S \pi_{\mathcal{K}}(\langle s,t\rangle_B)\eta\rangle &= \langle s\otimes_B \pi_{\mathcal{K}}(\langle y,x\rangle_B)\xi, T(t\otimes_B \eta)\rangle \\
    &= \langle \langle s,y\rangle_A x \otimes_B \xi, T(t\otimes_B \eta)\rangle\\
    &=\langle x \otimes_B \xi, T(\langle y,s\rangle_A t \otimes_B \eta)\rangle\\
    &= \langle x \otimes_B \xi, T(y\otimes_B \pi_{\mathcal{K}}(\langle s,t\rangle_B) \eta)\rangle\\
    &= \langle \xi, S\pi_{\mathcal{K}}(\langle x,y\rangle_B) \pi_{\mathcal{K}}(\langle s,t\rangle_B) \eta\rangle,
\end{align*}
from which we deduce that
\begin{equation}\label{intertwine}
    \pi_{\mathcal{K}}(\langle x,y\rangle_B)S \zeta = S \pi_{\mathcal{K}}(\langle x,y\rangle_B)\zeta, \quad \zeta \in \mathcal{K}_0.
\end{equation}
If $\xi, \eta \in \mathcal{K}_0$ and $x,y \in X$, we therefore find
$$\langle x \otimes_B\xi, y\otimes_B S \eta\rangle= \langle \xi, \pi_{\mathcal{K}}(\langle x,y\rangle_B)S\eta\rangle = \langle \xi, S \pi_{\mathcal{K}}(\langle x,y\rangle_B)\eta \rangle = \langle x \otimes_B \xi, T(y\otimes_B \eta)\rangle,$$
whence 
\begin{equation}\label{identity}
    T(y\otimes_B \eta) = y\otimes_B S \eta, \quad y \in X, \quad \eta \in \mathcal{K}_0.
\end{equation}
In particular, since $T$ is isometric, we find for $\xi, \eta \in \mathcal{K}_0$ and $x,y\in X$ that 
$$\langle \eta, \pi_{\mathcal{K}}(\langle y,x\rangle_B)\xi\rangle= \langle y\otimes_B \eta, x \otimes_B \xi\rangle = \langle y\otimes_B S\eta, x \otimes_B S\xi\rangle = \langle S\eta, \pi_{\mathcal{K}}(\langle y,x\rangle_B)S\xi\rangle.$$
Since $[\langle X,X\rangle_B]^{\sigma\text{-strong}}=B$, we conclude that 
$$\langle \eta, \xi\rangle = \langle S\eta, S\xi\rangle, \quad \eta, \xi \in \mathcal{K}_0.$$
Therefore, $S$ is isometric and extends uniquely to an isometry $\tilde{S}\in B(\mathcal{K}).$ By \eqref{intertwine}, we deduce that $\tilde{S}\in \pi_{\mathcal{K}}(B)'= \rho_\mathcal{K}(C)$. Then \eqref{identity}  shows that $T = \rho_\boxtimes(c)$ for some $c\in C$. We conclude that $\pi_{\boxtimes}(A)' = \rho_\boxtimes(C)$.

It remains to check that $\pi_{\boxtimes}$ and $\rho_{\boxtimes}$ are faithful. But if $a\in A$ and $\pi_{\boxtimes}(a) = 0$, then 
$$0 = \langle x\otimes_B \xi, ay\otimes_B \eta\rangle = \langle \xi, \pi_\mathcal{K}(\langle x,ay\rangle_B) \eta\rangle,\quad x,y \in Y, \quad \xi, \eta \in \mathcal{K}.$$
By faithfulness of $\pi_\mathcal{K}$, it follows that $\langle x,ay\rangle_B = 0$ for all $x, y \in X$ whence $aX = 0$. Since $[\langle X,X\rangle_A]^{\sigma\text{-strong}}=A$, this implies $a= 0$. Thus, $\pi_\boxtimes$ is faithful. Similarly, if $\rho_\boxtimes(c)= 0$ for some $c\in C$, then $\pi_\mathcal{K}(\langle x, y\rangle_B) \rho_\mathcal{K}(c)= 0$ for all $x,y \in X$. Since $\rho_\mathcal{K}$ is faithful, we conclude that $c = 0$. Thus, $\rho_\boxtimes$ is faithful.
\end{proof}

We call the von Neumann algebras $A$ and $B$ \emph{$W^*$-Morita equivalent} if there exists a Morita $A$-$B$-correspondence, or equivalently if there exists an equivalence $A$-$B$-bimodule. 

\subsection{Locally compact quantum groups}
A \emph{Hopf-von Neumann algebra} is a pair $(M, \Delta)$ where $M$ is a von Neumann algebra and $\Delta: M \to M \ovot M$ a unital, normal, isometric $*$-homomorphism such that $(\Delta \otimes \id)\circ \Delta = (\id \otimes \Delta)\circ \Delta$.

A \emph{locally compact quantum group} \cites{KV00, KV03,VV03} $\G$ is a Hopf-von Neumann algebra $(L^\infty(\G), \Delta)$ for which there exist normal, semifinite, faithful weights $\Phi, \Psi: L^\infty(\G)_+\to [0, \infty]$ such that $(\id\otimes \Phi)\Delta(x)= \Phi(x)1$ for all $x\in \mathscr{M}_\Phi^+$ and $(\Psi\otimes \id)\Delta(x)= \Psi(x)1$ for all $x\in \mathscr{M}_\Psi^+$. The von Neumann algebra predual $L^1(\G):= L^\infty(\G)_*$ is a completely contractive Banach algebra for the multiplication 
$$\mu\star \nu:= (\mu\otimes \nu)\circ \Delta, \quad \mu, \nu \in L^1(\G).$$

Without any loss of generality, we may (and will) assume that $L^\infty(\G)$ is faithfully represented in standard form on the Hilbert space $L^2(\G)$. To simplify notation, the standard anti-$*$-representation $\rho_{L^\infty(\G)}$ will simply be denoted by $\rho_\G$.

With respect to the weights $\Phi, \Psi$, we then have GNS-maps
$$\Lambda_\Phi: \mathscr{N}_\Phi\to L^2(\G), \quad \Lambda_\Psi: \mathscr{N}_\Psi\to L^2(\G).$$

Fundamental to the theory of locally compact quantum groups are the unitaries 
\[
V,W \in B(L^2(\G)\otimes L^2(\G)),
\]
called respectively \emph{right} and \emph{left} regular unitary representation. They are uniquely characterized by the identities 
\[
(\id\otimes \omega)(V) \Lambda_{\Psi}(x) = \Lambda_{\Psi}((\id\otimes \omega)\Delta(x)),
\qquad \omega \in L^1(\G),x\in \msN_{\Psi},
\]
\[
 (\omega \otimes \id)(W^*)\Lambda_{\Phi}(x) = \Lambda_{\Phi}((\omega\otimes \id)\Delta(x)),\qquad \omega \in L^1(\G),x\in \msN_{\Phi}.
\]
They are \emph{multiplicative unitaries} \cite{BS93} meaning that
\[
V_{12}V_{13}V_{23} = V_{23}V_{12},\qquad W_{12}W_{13}W_{23}= W_{23}W_{12},
\]
and they implement the coproduct of $L^\infty(\G)$ in the sense that
$$\label{EqComultImpl}
W^*(1\otimes x)W = \Delta(x) = V(x\otimes 1)V^*,\qquad x\in L^\infty(\G).$$
Moreover, we have
$$C_0(\G):=[(\omega\otimes \id)(V) \mid \omega \in B(L^2(\G))_*] = [(\id\otimes \omega)(W) \mid \omega \in B(L^2(\G))_*],$$
which is a $\sigma$-weakly dense C$^*$-subalgebra of $L^\infty(\G)$. Then $\Delta(C_0(\G))\subseteq M(C_0(\G)\otimes C_0(\G))$. 

We can also define the von Neumann algebra
$L^\infty(\hat{\G}) = [(\omega\otimes \id)(W) \mid \omega \in L^1(\G)]^{\sigma\textrm{-weak}}$
with coproduct
$\hat{\Delta}(\hat{x}) = \Sigma W(\hat{x}\otimes 1)W^*\Sigma$ where $\hat{x}\in L^\infty(\hat{\G})$. 
Then the pair $(L^\infty(\hat{\G}), \hat{\Delta})$ defines the locally compact quantum group $\hat{\G}$, which is called the dual of $\G$. Similarly, we can also define the von Neumann algebra $L^\infty(\check{\G}) = [(\id\otimes \omega)(V) \mid \omega \in L^1(\G)]^{\sigma\textrm{-weak}}$
with coproduct 
$\check{\Delta}(\check{x}) = V^*(1\otimes \check{x})V$ for $\check{x}\in L^\infty(\check{\G})$.
Then the pair $(L^\infty(\check{\G}), \check{\Delta})$ defines a locally compact quantum group $\check{\G}$. We have $L^\infty(\hat{\G})= L^\infty(\check{\G})'$. If $G$ is a locally compact group, we have
$L^\infty(\hat{G}) = \mathscr{L}(G)$ and $L^\infty(\check{G}) = \mathscr{R}(G)$, the left resp. right group von Neumann algebra associated with $G$. Therefore, $\hat{\G}$ and $\check{\G}$ should be thought of as a left and a right version of the dual locally compact quantum group. The von Neumann algebras $L^\infty(\check{\G})$ and $L^\infty(\hat{\G})$, acting naturally on $L^2(\G)$, are in standard form.

The left multiplicative unitary of $\check{\G}$ will be denoted by $\check{W}$ and the right multiplicative unitary of $\check{\G}$ will be denoted by $\check{V}$. We have $\check{W}= V$ and 
$$V \in L^\infty(\check{\G})\ovot L^\infty(\G), \quad W \in L^\infty(\G)\ovot L^\infty(\hat{\G}), \quad \check{V}\in L^\infty(\G)'\ovot L^\infty(\check{\G}).$$
In fact, we have e.g.\ $V\in M(C_0(\check{\G})\otimes C_0(\G))$ and similarly for $W$ and  $\check{V}$. For future use, we will also need the following notations:
\begin{align*}
    &\Delta_l: B(L^2(\G)) \to  L^\infty(\G)\ovot B(L^2(\G)): x \mapsto W^*(1\otimes x)W,\\
    &\Delta_r: B(L^2(\G))\to B(L^2(\G))\ovot L^\infty(\G): x \mapsto V(x\otimes 1)V^*,\\
    &\check{\Delta}_r: B(L^2(\G))\to B(L^2(\G)) \ovot L^\infty(\check{\G}): x \mapsto \check{V}(x\otimes 1)\check{V}^*.
\end{align*}
A locally compact quantum group $\G$ is called \emph{compact} if $C_0(\G)$ is unital, and we then write $C_0(\G)= C(\G)$. The canonical dense Hopf $^*$-subalgebra of $C(\G)$ will be denoted by $\mathcal{O}(\G)$. In that case, the left Haar weight $\Phi$ is a normal state and $\Phi = \Psi$ (after appropriate normalizations).

A locally compact quantum group $\G$ is called \emph{discrete} if $\check{\G}$ is compact and we then write $L^\infty(\G)= \ell^\infty(\G)$ and $C_0(\G)= c_0(\G)$. If $\G$ is a discrete quantum group, there is a unique normal state $\epsilon \in \ell^1(\G):= L^1(\G)$, called \emph{counit}, such that 
$(\epsilon \otimes \id)\Delta = \id = (\id \otimes \epsilon)\Delta$.

A \emph{unitary representation} of a locally compact quantum group $\G$ on a Hilbert space consists of a pair $(\mathcal{H}, U)$ where $\mathcal{H}$ is a Hilbert space and $U\in B(\mathcal{H})\ovot L^\infty(\G)$ a unitary satisfying $(\id \otimes \Delta)(U) = U_{12}U_{13}$. We write $(\mathcal{H}, U)\in \Rep(\G)$. For example, we have $V \in \Rep(\G)$ and  $W_{21}\in \Rep(\G)$. Given a Hilbert space $\mathcal{H}$, we write $\mathbb{I}= 1 \otimes 1 \in B(\mathcal{H})\ovot L^\infty(\G)$ for the trivial representation. Given $U \in B(\mathcal{H})\ovot L^\infty(\G)$ and $\omega \in L^1(\G)$, we write $U(\omega):= (\id \otimes \omega)(U)$.

Consider the modular conjugation $J:= J_{L^\infty(\G)}$ associated to the standard representation $L^\infty(\G)\subseteq B(L^2(\G))$ and the modular conjugation $\check{J}:= J_{L^\infty(\check{\G})}$ associated to the standard representation $L^\infty(\check{\G})\subseteq B(L^2(\G))$. We have 
$\check{J}L^\infty(\G)\check{J}=L^\infty(\G)$
so we obtain the anti-$*$-morphism
$$R: L^\infty(\G)\to L^\infty(\G): x \mapsto \check{J}x^*\check{J}.$$
We call $R$ the \emph{unitary antipode} of $\G$. There is a canonical unimodular complex number $c\in \mathbb{C}$ such that 
$$c \check{J}J=\overline{c}J \check{J}.$$
We write $u_\G:= c \check{J} J$ for the associated self-adjoint unitary. The following identities will be useful:
\begin{align*}
    (\check{J}\otimes J)W(\check{J}\otimes J)=W^*, \quad (J\otimes \check{J})V(J\otimes \check{J})=V^*, \quad (u_\G\otimes 1)V(u_\G\otimes 1)=W_{21}.
\end{align*}
\subsection{Dynamical von Neumann algebras and crossed products}
Let $\G$ be a locally compact quantum group.
A (right) \emph{$\G$-$W^*$-algebra} is a pair $(A, \alpha)$ such that $A$ is a von Neumann algebra and $\alpha: A \to A\ovot L^\infty(\G)$ is an injective, unital, normal $*$-homomorphism satisfying the coaction property $(\alpha\otimes \id)\circ \alpha = (\id \otimes \Delta)\circ \alpha$. We sometimes denote this with $\alpha: A\curvearrowleft\G$. We write
$A^{\alpha} = \{a\in A: \alpha(a)= a \otimes 1\}$
for the von Neumann subalgebra of fixed points of $(A, \alpha)$. The trivial $\G$-action on a von Neumann algebra $A$ will always be denoted by $\tau$, i.e.\ $\tau(a)= a \otimes 1$ for $a\in A$. 

If $(A, \alpha)$ and $(B, \beta)$ are $\G$-$W^*$-algebras, then a unital completely positive (=ucp) map $T: A \to B$ (not necessarily normal) is said to be $\G$-equivariant if $(T \otimes \id)\circ \alpha = \beta \circ T$, where $T\otimes \id: A \ovot L^\infty(\G)\to B \ovot L^\infty(\G)$ is the unique ucp map such that
$$(\id \otimes \omega)((T\otimes \id)(z)) = T((\id \otimes \omega)(z)), \quad z \in A \ovot L^\infty(\G), \quad \omega \in L^1(\G).$$
If we want to emphasize the actions, we say that $T: (A, \alpha)\to (B, \beta)$ is a $\G$-equivariant map.

Given a $\G$-$W^*$-algebra $(A, \alpha)$,  we define the \emph{crossed product von Neumann algebra}
$$A\rtimes_\alpha \G= [\alpha(A)(1\otimes L^\infty(\check{\G}))]''.$$
We have the alternative description
$$A\rtimes_\alpha \G = \{z \in A \ovot B(L^2(\G)): (\alpha\otimes \id)(z) =(\id \otimes \Delta_l)(z)\}.$$
It is easy to verify that 
$$(\id \otimes \check{\Delta}_r)(A\rtimes_\alpha \G)\subseteq (A\rtimes_\alpha \G)\ovot L^\infty(\check{\G}),$$
so that $(A\rtimes_\alpha \G, \id \otimes \check{\Delta}_r)$ becomes a $\check{\G}$-$W^*$-algebra. This $\check{\G}$-action on the crossed product will often be implicitly understood. We have that $(A\rtimes_\alpha \G)^{\id \otimes \check{\Delta}_r}= \alpha(A)$. Given any von Neumann algebra $A$, we have $A\rtimes_\tau \G = A\ovot L^\infty(\check{\G})$.

If $(A, \alpha)$ is a $\G$-$W^*$-algebra where $A$ is standardly represented on a Hilbert space $\mathcal{H}$, there exists a canonical unitary $\G$-representation $U_\alpha \in B(\mathcal{H})\ovot L^\infty(\G)$ such that $\alpha(a)= U_\alpha(a\otimes 1)U_\alpha^*$ for all $a\in A$ \cite{Va01}. We call $U_\alpha$ the \emph{unitary implementation} of $\alpha$. More precisely, we can identify $L^2(A\rtimes_\alpha \G)= L^2(A)\otimes L^2(\G)$ in such a way that the standard representation $\pi_{A\rtimes_\alpha \G}$ corresponds to $\pi_A\otimes \id: A\rtimes_\alpha \G \to B(L^2(A)\otimes L^2(\G))$ and such that the anti-$*$-representation $\rho_{A\rtimes_\alpha \G}: A\rtimes_\alpha \G \to B(L^2(A)\otimes L^2(\G))$ is given by
$$\rho_{A\rtimes_\alpha \G}(\alpha(a)) = \rho_A(a)\otimes 1, \quad \rho_{A\rtimes_\alpha \G}(1\otimes \check{y})=U_\alpha(1\otimes \rho_{L^\infty(\check{\G})}(\check{y}))U_\alpha^*, \quad a\in A, \quad \check{y}\in L^\infty(\check{\G}).$$

\subsection{Equivariant correspondences} Let $\G$ be a locally compact quantum group. We recall some of the theory from \cite{DCDR24a}. Definition \ref{correspondence_definition} is generalized to the equivariant setting as follows:

Let $(A, \alpha)$ and $(B, \beta)$ be two $\G$-$W^*$-algebras. 
A $\G$-$A$-$B$-correspondence consists of a Hilbert space $\mathcal{H}$ together with the following data:
\begin{itemize}\setlength\itemsep{-0.5em}
    \item A unitary $\G$-representation $U\in B(\mathcal{H})\ovot L^\infty(\G)$.
    \item A unital, normal $*$-homomorphism $\pi: A\to B(\mathcal{H})$ such that $(\pi\otimes \id)\alpha(a) = U(\pi(a)\otimes 1)U^*$ for all $a\in A$.
    \item A unital, normal anti-$*$-homomorphism $\rho: B \to B(\mathcal{H})$ such that $(\rho\otimes R)\beta(b) = U^*(\rho(b)\otimes 1)U$ for all $b\in B$.
\end{itemize}

Given two $\G$-$A$-$B$-correspondences $\mathcal{H}$ and $\mathcal{K}$, a bounded linear map $x: \mathcal{H}\to \mathcal{K}$ is called \emph{intertwiner} of correspondences if 
$$x\pi_{\mathcal{H}}(a) = \pi_{\mathcal{K}}(a)x, \quad x\rho_{\mathcal{H}}(b) = \rho_{\mathcal{K}}(b)x, \quad x U_{\mathcal{H}}(\omega) = U_{\mathcal{K}}(\omega)x, \quad a\in A, \quad b\in B, \quad \omega \in L^1(\G).$$
The set of such intertwiners is denoted by ${}_A\mathscr{L}_B^\G(\mathcal{H}, \mathcal{K})$. If $\G$ is the trivial group or if one of the von Neumann algebras $A$ or $B$ is equal to $\C$, we omit it from the notation. 

In this way, we obtain the $W^*$-category $\Corr^\G(A,B)$ of $\G$-$A$-$B$-correspondences. One can `compose' equivariant correspondences:
Given $\mathcal{H}\in \Corr^\G(A,B)$ and $\mathcal{G}\in \Corr^\G(B,C)$, we can (see subsection \ref{corr}) form
$$\mathcal{H}\boxtimes_B\mathcal{G}:= \mathscr{L}_B(L^2(B),\mathcal{H})\otimes_B\mathcal{G}\in \Corr(A,C).$$

Moreover, as is proven in \cite{DCDR24a}*{Proposition 5.6}, this correspondence carries a canonical $\G$-representation $U_{\boxtimes}\in B(\mathcal{H}\boxtimes_B \mathcal{G})\ovot L^\infty(\G)$ such that $$(\mathcal{H}\boxtimes_B \mathcal{G}, \pi_\boxtimes, \rho_\boxtimes, U_\boxtimes)\in \Corr^\G(A,C).$$

\section{Characterizations of equivariant Morita equivalence}

In this section, we develop the equivariant Morita theory. We give several equivalent characterizations of the notion of equivariant Morita equivalence, and we discuss examples of equivariant Morita equivalences that appear naturally in the theory of locally compact quantum groups. Throughout this entire section, we fix a locally compact quantum group $\G$ and $\G$-$W^*$-algebras $(A,\alpha)$ and $(B, \beta)$. 

We will use the following generalization of Definition \ref{definitionn}, introduced first in \cite{DCDR24a}, as our starting point:
\begin{Def}\label{definition Morita}
 A $\G$-Morita correspondence between $(A, \alpha)$ and $(B, \beta)$ consists of $(\mathcal{H}, \pi, \rho, U)\in \Corr^\G(A,B)$ such that $\pi, \rho$ are faithful and $\pi(A)'= \rho(B)$.
    If such a $\G$-Morita correspondence exists, the $\G$-$W^*$-algebras $(A, \alpha)$ and $(B, \beta)$ are called $\G$-$W^*$-Morita equivalent, and we write $(A, \alpha)\sim_\G (B, \beta)$.
\end{Def}

If $A,B$ are Morita equivalent von Neumann algebras, then we simply write $A\sim B$. This is the situation where $\G$ is the trivial group. 

To justify the terminology in Definition \ref{definition Morita}, we need to prove the following:

\begin{Prop}
    $\sim_\G$ is an equivalence relation.
\end{Prop}
\begin{proof} Considering the $\G$-Morita correspondence $(L^2(A), \pi_A, \rho_A, U_\alpha)\in \Corr^\G(A,A)$, we deduce that $\sim_\G$ is reflexive. If $\mathcal{H} \in \Corr^\G(A,B)$ is a $\G$-Morita correspondence, then $\overline{\mathcal{H}}\in \Corr^\G(B,A)$ \cite{DCDR24a}*{Section 5.1} is also a $\G$-Morita correspondence, proving the symmetry of $\sim_\G$. Suppose that $\mathcal{H}\in \Corr^\G(A,B)$ and $\mathcal{K}\in \Corr^\G(B,C)$ are $\G$-Morita correspondences. By Proposition \ref{technical}, 
$\mathcal{H}\boxtimes_B \mathcal{K} \in \Corr^\G(A,C)$ is a $\G$-Morita correspondence. This shows that $\sim_\G$ is transitive.
\end{proof}

The following trivial observation will be useful in the sequel:

\begin{Prop}[Amplification]\label{amplification} Let $M,N$ be $W^*$-algebras. Then
$$M\sim N, \quad (A, \alpha)\sim_\G (B, \beta)\implies (M\ovot A, \id \otimes \alpha)\sim_\G (N \ovot B, \id \otimes \beta).$$
\end{Prop}
\begin{proof}
    Let $(\mathcal{H},\pi_\mathcal{H}, \rho_\mathcal{H}) \in \operatorname{Corr}(M,N)$ be a Morita correspondence and let $(\mathcal{K}, \pi_\mathcal{K}, \rho_\mathcal{K}, U) \in \operatorname{Corr}^\G(A ,B)$ be a $\G$-Morita correspondence. Then it is easily checked that
    $$(\mathcal{H}\otimes \mathcal{K}, \pi_\mathcal{H}\otimes \pi_\mathcal{K}, \rho_\mathcal{H}\otimes \rho_\mathcal{K}, U_{23})\in \operatorname{Corr}^\G(M\ovot A ,N \ovot B)$$
    is a $\G$-Morita correspondence.
\end{proof}

We now provide an alternative characterization of equivariant $W^*$-Morita equivalence, based on the linking algebra approach to Morita equivalence. First, let us recall that if $(E, \gamma)$ is a $\G$-$W^*$-algebra, then $E^\gamma:= \{x \in E: \gamma(x)= x \otimes 1\}$ is a von Neumann subalgebra of $E$, called the \emph{fixed point von Neumann subalgebra}. If $p\in E^\gamma$ is a projection, then $\gamma(pEp)\subseteq (pEp)\ovot L^\infty(\G)$ and $(pEp, \gamma)$ becomes a $\G$-$W^*$-algebra.

\begin{Def}
    A linking $\G$-$W^*$-algebra between $(A, \alpha)$ and $(B, \beta)$ consists of the data $(E, \gamma, p, \kappa, \lambda)$ such that $(E, \gamma)$ is a $\G$-$W^*$-algebra, $p\in E^\gamma$ is a projection such that $p,p^\perp$ are full projections of $E$ and such that $\kappa: A \to pEp$ and $\lambda: B\to p^\perp Ep^\perp$ are $\G$-equivariant isomorphisms of von Neumann algebras.
\end{Def}

If $A,B$ are von Neumann algebras and $\mcH \in \Corr(A,B)$, we define $\mathscr{L}_B(L^2(B), \mcH)\ovot L^\infty(\G)$ to be the $\sigma$-weak closure of $\mathscr{L}_B(L^2(B), \mcH)\odot L^\infty(\G)$ inside $B(L^2(B) \otimes L^2(\G), \mcH \otimes L^2(\G))$. Such tensor products will occasionally appear throughout this paper, for example in the proof of the next result.

\begin{Theorem}\label{char1} 
    The following are equivalent:
    \begin{enumerate}\setlength\itemsep{-0.5em}
        \item There exists a $\G$-Morita correspondence $(\mathcal{H}, \pi, \rho, U)\in \Corr^\G(A,B)$.
        \item There exists a linking $\G$-$W^*$-algebra $(E, \gamma,p, \kappa, \lambda)$ between $(A, \alpha)$ and $(B, \beta)$.
    \end{enumerate}
\end{Theorem}

\begin{proof}
    $(1)\implies (2)$ The argument we present here is implicitly contained in \cite{DCDR24a}*{Section 5}.
    
    Let $(\mathcal{H}, \pi, \rho, U)\in \operatorname{Corr}^\G(A,B)$ such that $\pi$ and $\rho$ are faithful and $\pi(A)'= \rho(B)$. Let
    $$X:= \mathscr{L}_B(L^2(B), \mathcal{H})= \{x\in B(L^2(B), \mathcal{H})\mid\forall b \in B: x\rho_B(b) = \rho(b)x\}$$
    and consider the von Neumann algebra 
    $$E:= \begin{pmatrix}\pi(A) & X \\ X^* &\pi_B(B)\end{pmatrix}\subseteq B\begin{pmatrix}
        \mathcal{H}\\ L^2(B)
    \end{pmatrix}.$$
    In other words, $E$ is exactly the commutant of the image of the direct sum anti-$*$-representation
    $$B \to B(\mathcal{H}\oplus L^2(B)): b \mapsto \begin{pmatrix} \rho(b) & 0 \\ 0 & \rho_B(b)
    \end{pmatrix}.$$
    Put $p:= \begin{pmatrix}1 & 0 \\ 0 & 0\end{pmatrix}\in E$ and define $U:= U\oplus U_\beta \in B(\mathcal{H}\oplus L^2(B))\ovot L^\infty(\G)$. Define then the $\G$-action
    $$\gamma: E\to E \ovot L^\infty(\G): z \mapsto U(z\otimes 1)U^*.$$ 

We can identify
$$E\ovot L^\infty(\G) \cong \begin{pmatrix}
    \pi(A) \ovot L^\infty(\G) & X \ovot L^\infty(\G)\\ X^*\ovot L^\infty(\G) & \pi_B(B)\ovot L^\infty(\G) \end{pmatrix}: \begin{pmatrix}
        x_{11} & x_{12} \\ x_{21} & x_{22} \end{pmatrix}
    \otimes g \mapsto \begin{pmatrix}
        x_{11} \otimes g& x_{12}\otimes g \\ x_{21}\otimes g & x_{22}\otimes g
    \end{pmatrix},
$$
and under this identification we see that
$$\gamma \begin{pmatrix}
        x_{11} & x_{12} \\ x_{21} & x_{22} \end{pmatrix} = \begin{pmatrix}
        U(x_{11}\otimes 1)U^* & U(x_{12}\otimes 1)U_\beta^* \\ U_\beta(x_{21}\otimes 1)U^* & U_\beta(x_{22}\otimes 1)U_\beta^* \end{pmatrix}.$$
        It is therefore clear that $p \in E^\gamma$. Further, we note that the map
        $$\kappa: A \to pEp: a\mapsto \begin{pmatrix}\pi(a) & 0 \\ 0 & 0\end{pmatrix}$$
        is a $\G$-equivariant isomorphism, since 
        \begin{align*}
            (\kappa \otimes \id)\alpha(a)&= \begin{pmatrix}(\pi\otimes \id)(\alpha(a)) & 0 \\ 0 & 0\end{pmatrix}= \begin{pmatrix}U(\pi(a)\otimes 1)U^* & 0 \\ 0 & 0\end{pmatrix}= \gamma(\kappa(a)), \quad a \in A,
        \end{align*}
so that $A \cong_\G pEp$. Similarly, $B \cong_\G p^\perp Ep^\perp.$ Using the fact that $X$ is an equivalence $A$-$B$-bimodule, straightforward matrix calculations show that $p$ and $p^\perp$ are full projections.
        
    $(2)\implies (1)$ Consider a linking $\G$-$W^*$-algebra $(E, \gamma, p, \mu, \nu)$. To simplify notation, we write $C:= pEp$ and $D:= p^\perp E p^\perp$. We note that $X:=pEp^\perp$ is an equivalence $C$-$D$-bimodule for the inner products
$$\langle x,y\rangle_{C} = xy^*, \quad \langle x,y\rangle_{D} = x^*y, \quad x,y \in X=pEp^\perp.$$
Since $p\in E^\gamma$, the coaction $\gamma: E \to E \ovot L^\infty(\G)$ restricts to a coaction $\delta: pEp^\perp \to (pEp^\perp)\ovot L^\infty(\G)$.

We then define the isometry
$$U_\boxtimes: (X\otimes_{D} L^2(D))\otimes L^2(\G) \to (X\ovot L^\infty(\G))\otimes_{D \ovot L^\infty(\G)} (L^2(D) \otimes L^2(\G))\cong (X \otimes_D L^2(D))\otimes L^2(\G)$$
by
$$U_\boxtimes((x\otimes_D \xi)\otimes \eta) = \delta(x) \otimes_{D\ovot L^\infty(\G)}U_{\gamma\vert_D}(\xi \otimes \eta), \quad x \in X, \quad \xi \in L^2(D), \quad \eta \in L^2(\G).$$
Exactly as in \cite{DCDR24a}*{Proposition 5.6}, we argue that $U_\boxtimes$ defines a unitary $\G$-representation such that 
$$(X\otimes_D L^2(D), \pi_{\boxtimes}, \rho_{\boxtimes}, U_{\boxtimes})\in \Corr^\G(C,D).$$
By Proposition \ref{technical}, it is a $\G$-Morita correspondence. Therefore, we see that
$$A \cong_\G C \sim_\G D \cong_\G B,$$
which finishes the proof.
\end{proof}

\begin{Rem}
    There is another way to prove the implication $(2)\implies (1)$ of Theorem \ref{char1}, which we now sketch. In the non-equivariant setting, the idea for the construction we present here is given in \cite{DC09}*{Section 5.5}.
    
    Indeed, let $(E, \gamma)$ be a $\G$-$W^*$-algebra and let $p\in E^\gamma$ be a projection such that $p$ and $p^\perp$ are full in $E$. Consider the $\G$-$W^*$-algebras $C:= pEp$ and $D:= p^\perp E p^\perp$. Consider the Hilbert space $$\mathcal{H}:=\pi_E(p) \rho_E(p^\perp)L^2(E).$$
   There is a natural $C$-$D$-correspondence $(\mathcal{H}, \pi, \rho)$, which can be checked to be a Morita $C$-$D$-correspondence, making use of the fullness of $p$ and $p^\perp$. Consider the unitary implementation $U_\gamma\in B(L^2(E))\ovot L^\infty(\G)$. Since $p\in E^\gamma$, $\mathcal{H}$ is invariant for the $\G$-representation $U_\gamma$ and we obtain the unitary $\G$-subrepresentation $U\in B(\mathcal{H})\ovot L^\infty(\G)$. Then it can be checked that $(\mathcal{H}, \pi, \rho, U)\in \Corr^\G(C,D)$ is a $\G$-Morita correspondence. To make the connection with the proof of Theorem \ref{char1} clear, we note that there is a unitary isomorphism
   $$X \otimes_D L^2(D)\to \mathcal{H}: x\otimes_D\xi \mapsto \pi_E(x)\xi$$
   of $\G$-$C$-$D$-correspondences, 
   where we isometrically identify $L^2(D)\subseteq L^2(E)$. 

   For concreteness, fix an nsf weight $\varphi$ on $C$ and an nsf weight $\psi$ on $D$. Then we can construct the `balanced' nsf weight
   $\theta: E^+ \to [0, \infty]$ by $\theta(x)= \varphi(pxp)+ \psi(p^\perp xp^\perp).$
   In this way, using the GNS-construction w.r.t. the weight $\theta$, we have a concrete realization of the Hilbert space $L^2(E)$. Arguing more or less as in \cite{Tak03}*{VIII.3}, we can then use this concrete realization to verify the details of the claims that were made in this remark.
\end{Rem}

Next, we give another characterization of equivariant Morita equivalence.  First, we discuss the prototypical example of an equivariant Morita equivalence:

\begin{Exa}\label{main example}
    Let $(B, \beta)$ be a $\G$-$W^*$-algebra, $p\in B$ a full projection and $P\in B \ovot L^\infty(\G)$ satisfying
    $$(\id \otimes \Delta)(P) = (P\otimes 1) (\beta \otimes \id)(P), \quad P^*P=\beta(p), \quad PP^* = p\otimes 1.$$
    Then the following statements are true:
    \begin{enumerate}\setlength\itemsep{-0.5em}
        \item The assignment  $\beta_P: pBp \to pBp\ovot L^\infty(\G): x \mapsto P\beta(x)P^*$
    defines a $\G$-action $pBp \curvearrowleft \G$. 
    \item $(B, \beta)\sim_\G (pBp, \beta_P).$
    \end{enumerate}
\end{Exa}
\begin{proof} Let us first remark that the assumption $PP^*= p\otimes 1$ implies that $(p\otimes 1)P= P$.

Statement $(1)$ follows by straightforward calculations. Define the element $$X:= (\pi_B\otimes \id)(P) U_\beta \in B(\pi_B(p)L^2(B))\ovot L^\infty(\G).$$ Then $X$ is a unitary $\G$-representation on the Hilbert space $\pi_B(p)L^2(B)$, since we have
\begin{align*}
    (\id \otimes \Delta)(X)&= (\pi_B\otimes \id\otimes \id)(\id \otimes \Delta)(P) U_{\beta, 12}U_{\beta, 13}\\
    &= (\pi_B\otimes \id)(P)_{12} (\pi_B\otimes \id \otimes \id)(\beta \otimes \id)(P) U_{\beta,12} U_{\beta, 13}\\
    &= (\pi_B\otimes \id)(P)_{12} U_{\beta, 12}(\pi_B\otimes \id)(P)_{13}U_{\beta, 12}^* U_{\beta,12}U_{\beta, 13}= X_{12} X_{13}.
\end{align*}
We then easily verify that the following data defines a $\G$-Morita correspondence on the Hilbert space $\pi_B(p)L^2(B)$:
\begin{itemize}\setlength\itemsep{-0.5em}
    \item The unital, normal $*$-representation $\pi: pBp\to B(\pi_B(p)L^2(B)): x \mapsto \pi_B(x)$.
    \item The unital, normal anti-$*$-representation $\rho: B\to B(\pi_B(p)L^2(B)): b \mapsto \rho_B(b)$.
    \item The unitary $\G$-representation $X$.
\end{itemize}
We just remark that fullness of $p$ is necessary to ensure that $\rho$ is faithful. Therefore, the statement $(2)$ follows.
\end{proof}

% Consider the von Neumann algebra
 %   \begin{align*}
  %      E:=\begin{pmatrix}
   %     \pi_B(pBp) & \pi_B(pB) \\ \pi_B(Bp) & \pi_B(B)
    %\end{pmatrix}\subseteq B\begin{pmatrix}
     %   \pi_B(p)L^2(B)\\ L^2(B)
    %\end{pmatrix}
    %\end{align*}
    %and define
%$$\gamma: E \to E \ovot L^\infty(\G): \begin{pmatrix}
 %       x_{11} & x_{12} \\ x_{21} & x_{22} \end{pmatrix} \mapsto\begin{pmatrix}
  %   X(x_{11}\otimes 1)X^* & X(x_{12}\otimes 1) U_\beta^* \\ U_\beta(x_{21}\otimes 1)X & U_\beta(x_{22}\otimes 1)U_\beta^* \end{pmatrix}.$$
   %    Routine verifications show that $(E, \gamma)$ becomes a $\G$-$W^*$-linking algebra between $(pBp, \beta_U)$ and $(B, \beta)$.\end{proof}

This example is so important because (up to an amplification), every $\G$-$W^*$-Morita equivalence arises as in the above example. To prove this, we will use the following elementary fact of modular theory: If $\mathcal{H}$ is a Hilbert space and $x\in B(\mathcal{H})$, we write $x^T\in B(\overline{\mathcal{H}})$ for the operator given by $x^T \overline{\xi}= \overline{x^*\xi}.$ We have $L^2(B(\mathcal{H}))= \overline{\mathcal{H}}\otimes \mathcal{H}$ in a way that the left standard representation is given by
$\pi_{B(\mathcal{H})}(x)= 1_{\overline{\mathcal{H}}}\otimes x$ for $x\in B(\mathcal{H})$ and such that the modular conjugation is given by $J(\overline{\xi}\otimes \eta)= \overline{\eta}\otimes \xi$ for $\xi, \eta \in \mathcal{H}$. Consequently, the standard anti-$*$-representation is given by $\rho_{B(\mathcal{H})}(x) = x^T \otimes 1_{\mathcal{H}}$ for $x\in B(\mathcal{H})$.

\begin{Theorem}\label{char2} The following are equivalent:
      \begin{enumerate}\setlength\itemsep{-0.5em}
        \item $(A, \alpha)\sim_\G (B, \beta)$.
        \item There exists a Hilbert space $\mathcal{K}$, a full projection $p\in B(\mathcal{K})\ovot B$, $P\in B(\mathcal{K})\ovot B \ovot L^\infty(\G)$ and a $*$-isomorphism $\phi: p(B(\mathcal{K})\ovot B)p\to A$ such that the following conditions are satisfied:
    \begin{enumerate}\setlength\itemsep{-0.5em}
        \item $P^*P=(\id \otimes \beta)(p)$.
        \item $PP^* = p\otimes 1$.
        \item $(\id \otimes \id \otimes \Delta)(P) = (P\otimes 1)(\id \otimes \beta \otimes \id)(P).$
        \item $\alpha(\phi(z)) = (\phi\otimes \id)(P(\id \otimes \beta)(z)P^*)$ for all $z\in p(B(\mathcal{K})\ovot B)p.$
    \end{enumerate}
    \end{enumerate}
\end{Theorem}
\begin{proof} $(1)\implies (2)$ Let $(\mathcal{H}, \pi, \rho, U)\in \operatorname{Corr}^\G(A,B)$ be a $\G$-Morita correspondence. There exists a set $I$ and a projection $p\in B(\ell^2(I))\ovot B$ such that 
$$(\mathcal{H}, \rho)\cong ((\id \otimes \pi_B)(p)(\ell^2(I)\otimes L^2(B)), \tilde{\rho})$$
as right $B$-representations, where $\tilde{\rho}(b) = 1\otimes \rho_B(b)$ for $b\in B$.

Put $C:= B(\ell^2(I))\ovot B$, which we view as a $\G$-$W^*$-algebra for the amplified $\G$-action $\id \otimes \beta: C \to C \ovot L^\infty(\G)$. We prove now that $p\in C$ is full. Consider the two-sided $\sigma$-weakly closed ideal
$J = [CpC]^{\sigma\text{-weak}}.$
There is a central projection $r\in B$ such that $J = (1\otimes r)C$. We have
$$\tilde{\rho}(r^\perp)(\id \otimes \pi_B)(p)= (1\otimes \rho_B(r^\perp))(\id \otimes \pi_B)(p)= (\id \otimes \pi_B)((1\otimes r^\perp)p)= 0,$$
so that by faithfulness of $\tilde{\rho}$ we find that $r^\perp= 0$, so that $p$ is full in $C$.

Consider the $\G$-Morita-correspondence $\overline{\ell^2(I)}\otimes L^2(B)\in \Corr^\G(B,C)$ defined by
    \begin{itemize}\setlength\itemsep{-0.5em}
        \item the normal $*$-representation $\pi'(b) = 1_{\overline{\ell^2(I)}}\otimes \pi_B(b), \quad b \in B,$
        \item the normal anti-$*$-representation $\rho'(z) = ({(-)}^T \otimes \rho_B)(z), \quad z \in C,$
        \item the unitary $\G$-representation $U_{\beta, 23}\in B(\overline{\ell^2(I)}\otimes L^2(B))\ovot L^\infty(\G)$.
    \end{itemize}
We find
\begin{align*}
    \mathcal{H}\boxtimes_B (\overline{\ell^2(I)}\otimes L^2(B)) &\cong \overline{\ell^2(I)}\otimes \mathcal{H}\\
    &\cong \overline{\ell^2(I)}\otimes (\id \otimes \pi_B)(p)(\ell^2(I)\otimes L^2(B))\\
    &= (\pi_{B(\ell^2(I))}\otimes \pi_B)(p)(\overline{\ell^2(I)}\otimes \ell^2(I)\otimes L^2(B))\\
    &= \pi_{B(\ell^2(I))\ovot B}(p)L^2(B(\ell^2(I))\ovot B) = \pi_C(p)L^2(C)
\end{align*}
as right $C$-representations. Therefore, by structure transport, the Hilbert space $\mathcal{G}:=\pi_C(p)L^2(C)$ with its natural right $C$-action is endowed with the structure of a $\G$-$A$-$C$-Morita correspondence. The fact that $\mathcal{G}$ becomes a $\G$-$A$-$C$-Morita correspondence then implies that 
$$A \cong \pi_\mathcal{G}(A)= \rho_\mathcal{G}(C)' = (\pi_C(p) \pi_C(C) \pi_C(p))\vert_{\pi_C(p)L^2(C)}\cong \pi_C(pCp)\cong pCp,$$
so we find a $*$-isomorphism $\phi: pCp \to A$.

Consider the linking $\G$-$W^*$-algebra $E:=\mathscr{L}_C(\mathcal{G}\oplus L^2(C))$ between $(A, \alpha)$ and $(C, \id \otimes \beta)$ (see the proof of Theorem \ref{char1}) endowed with its natural $\G$-action $\gamma: E\to E\ovot L^\infty(\G)$
and the natural $\G$-equivariant embeddings
$$\kappa: A \to E: a \mapsto \begin{pmatrix}
    \pi_\mathcal{G}(a) & 0 \\ 0 & 0
\end{pmatrix}, \quad \lambda: C \to E: c \mapsto \begin{pmatrix}
    0 & 0 \\ 0 & \pi_C(c)
\end{pmatrix}.$$
Evidently, we can view
\begin{align*}
    \mathscr{L}_C(\mathcal{G}\oplus L^2(C))= \begin{pmatrix}
        \pi_C(pCp) & \pi_C(pC) \\ \pi_C(Cp) & \pi_C(C)
    \end{pmatrix},
\end{align*}
 where the right hand side acts naturally on $\mathcal{G}\oplus L^2(C)$. We note that 
$$ \gamma\begin{pmatrix}
    0 & \pi_C(p) \\ 0 & 0
\end{pmatrix}  = \begin{pmatrix}
    0 & (\pi_C\otimes \id)(P) \\ 0 & 0
\end{pmatrix}$$
for some element $P\in C \ovot L^\infty(\G)$ satisfying $P=(p\otimes 1)P$. Then note that 
$$\begin{pmatrix}
    0 & (\pi_C\otimes \id)(P) \\ 0 & 0
\end{pmatrix}^* \begin{pmatrix}
    0 & (\pi_C\otimes \id)(P) \\ 0 & 0
\end{pmatrix} = \gamma \begin{pmatrix}
    0 & 0 \\ 0 & \pi_C(p)
\end{pmatrix}= \begin{pmatrix}
    0 & 0 \\ 0 & (\pi_C \otimes \id)((\id \otimes \beta)(p))
\end{pmatrix}$$
implies the property $(a)$. Similarly, property $(b)$ is proven. Further,
\begin{align*}
    &\begin{pmatrix}
        0 & (\pi_C\otimes \id \otimes \id)((P\otimes 1)(\id \otimes \beta \otimes \id)(P))\\
        0 & 0
    \end{pmatrix}\\
    &= \begin{pmatrix}
        0 & (\pi_C\otimes \id)(P)\otimes 1\\ 0 & 0
    \end{pmatrix} \begin{pmatrix}
        0 & 0 \\ 0 &
         (\pi_C\otimes \id \otimes \id)(\id \otimes \beta \otimes \id)(P)
    \end{pmatrix}\\
    &= \left(\gamma\begin{pmatrix}
        0 & \pi_C(p)\\ 0 & 0
    \end{pmatrix}\otimes 1\right)(\gamma\otimes \id) \begin{pmatrix}
        0 & 0 \\ 0 & (\pi_C\otimes \id)(P)
    \end{pmatrix}\\
    &= (\gamma \otimes \id)\begin{pmatrix}
        0 & (\pi_C\otimes \id)(P) \\ 0 & 0
    \end{pmatrix}\\
    &= (\id \otimes \Delta)\begin{pmatrix}
        0 & (\pi_C\otimes \id)(P) \\ 0 & 0
    \end{pmatrix}= \begin{pmatrix}
        0 & (\pi_C\otimes \id \otimes \id)(\id_C \otimes \Delta)(P)\\
        0 & 0
    \end{pmatrix},
\end{align*}
from which $(c)$ follows. Next, note that for $c\in C$, 
\begin{align*}
    (\kappa \otimes \id)\alpha(\phi(pcp))&=  \gamma\begin{pmatrix}
        \pi_C(pcp) & 0 \\ 0 & 0
    \end{pmatrix}\\
    &=\gamma\left(\begin{pmatrix}
        0 & \pi_C(p) \\ 0 & 0
    \end{pmatrix} \begin{pmatrix}
        0 & 0 \\ 0 & \pi_C(c)
    \end{pmatrix} \begin{pmatrix}
        0 & 0 \\ \pi_C(p) & 0
    \end{pmatrix}\right)\\
    &= \begin{pmatrix}
        0 & (\pi_C\otimes \id)(P)\\ 0 & 0
    \end{pmatrix} \begin{pmatrix}
        0 & 0 \\ 0 & (\pi_C \otimes \id)((\id \otimes \beta)(c))
    \end{pmatrix} \begin{pmatrix}
        0 & 0 \\ (\pi_C\otimes \id)(P^*) & 0
    \end{pmatrix}\\
    &= \begin{pmatrix}
        (\pi_C\otimes \id)(P(\id \otimes \beta)(c)P^*) & 0 \\
        0 & 0
    \end{pmatrix}\\
    &= (\kappa \otimes \id)(\phi\otimes \id)(P(\id \otimes \beta)(c)P^*),
\end{align*}
from which $(d)$ follows. 

$(2)\implies (1)$ We have
$$(B, \beta) \sim_\G (B(\mathcal{K})\ovot B, \id \otimes \beta) \sim_\G (p(B(\mathcal{K})\ovot B)p, (\id \otimes \beta)_P) \stackrel{\phi}{\cong}_\G (A, \alpha),$$
where the first $\G$-$W^*$-Morita equivalence follows from Proposition \ref{amplification} and the second $\G$-$W^*$-Morita equivalence follows from Example \ref{main example}.
\end{proof}

We also record the following elementary fact:
\begin{Prop}\label{crossed product}
Let $\G$ be a locally compact quantum group and let $(A, \alpha)$ and $(B, \beta)$ be two $\G$-$W^*$-algebras.
    \begin{enumerate}\setlength\itemsep{-0.5em}
        \item $(A, \alpha)\sim_\G (B, \beta) \iff (A\rtimes_\alpha \G, \id \otimes \check{\Delta}_r)\sim_{\check{\G}}(B\rtimes_\beta \G, \id \otimes \check{\Delta}_r).$
        \item $(A, \alpha)\sim_\G (B, \beta) \implies (A\rtimes_\alpha \G, \id \otimes \Delta_r)\sim_{\G}(B\rtimes_\beta \G, \id \otimes \Delta_r).$
    \end{enumerate}
    Moreover, $\alpha: A \curvearrowleft \G$ is inner\footnote{Recall that $\alpha: A\curvearrowleft \G$ is called inner if there exists a unitary $\G$-representation $U\in A \ovot L^\infty(\G)$ such that $\alpha(a)= U(a\otimes 1)U^*$ for all $a\in A$.} if and only if $(A, \alpha)\sim_\G (A, \tau)$.
\end{Prop}
\begin{proof} The equivalence in $(1)$ is proven in \cite{DR24b}*{Proposition 2.1}.

If $\alpha: A \curvearrowleft \G$ is inner, then from \cite{DR24a}*{Lemma 6.1}, we know that $A \rtimes_\alpha \G \cong_{\check{\G}} A \ovot L^\infty(\check{\G}) = A \rtimes_\tau \G$. By $(1)$, we see that $(A, \alpha)\sim_{\G} (A, \tau)$. Alternatively, if $\alpha: A \curvearrowleft \G$ is inner, there exists a unitary $U\in A \ovot L^\infty(\G)$ such that $(\id \otimes \Delta)(U) = U_{12}U_{13}$ and $\alpha(a)= U(a\otimes 1)U^*$ for $a\in A$. Then $(L^2(A), \pi_A, \rho_A, (\pi_A\otimes \id)(U))$ is an explicit $\G$-Morita correspondence between $(A, \alpha)$ and $(A, \tau)$. Conversely, if $(A, \alpha)\sim_\G (A, \tau)$, it is an immediate consequence of Theorem \ref{char2} (apply it with $(B, \beta)= (A, \tau)$) that $(A, \alpha)$ is inner. 

Finally, if $(A, \alpha)\sim_\G (B, \beta)$, then 
$$(A\rtimes_\alpha \G, \id \otimes \Delta_r) \sim_\G (A\rtimes_\alpha \G, \tau) \sim_\G (B\rtimes_\beta \G, \tau) \sim_\G (B\rtimes_\beta \G, \id \otimes \Delta_r),$$
so that $(2)$ also follows.
\end{proof}

We end this section by considering important classes of examples of equivariant $W^*$-Morita equivalences.

Let us assume that $\mathbb{H}$ is a \emph{(Vaes-)closed quantum subgroup} of $\G$ \cite{DKSS12}, i.e.\ $\mathbb{H}$ is a locally compact quantum group and there exists a unital, normal, faithful $*$-homomorphism $\check{\gamma}: L^\infty(\check{\mathbb{H}})\to L^\infty(\check{\G})$ such that $(\check{\gamma}\otimes \check{\gamma})\circ \check{\Delta}^{\mathbb{H}}= \check{\Delta}^{\mathbb{G}}\circ \check{\gamma}.$ We also define the faithful normal $*$-homomorphism
$$\hat{\gamma}: L^\infty(\hat{\mathbb{H}})\to L^\infty(\hat{\G}): x \mapsto \check{J}_\G\check{\gamma}(\check{J}_{\mathbb{H}}x^* \check{J}_{\mathbb{H}})^*\check{J}_\G$$
satisfying $(\hat{\gamma}\otimes \hat{\gamma})\circ \hat{\Delta}^{\mathbb{H}} = \hat{\Delta}^\G \circ \hat{\gamma}.$ Define the unitaries
$$V_{\G, \mathbb{H}}:= (\check{\gamma} \otimes \id)(V_\mathbb{H})\in L^\infty(\check{\G})\bar{\otimes} L^\infty(\mathbb{H}), \quad W_{\mathbb{H}, \G}:=(\id \otimes \hat{\gamma})(W_{\mathbb{H}}) \in L^\infty(\mathbb{H})\ovot L^\infty(\hat{\G})$$
and the normal $*$-homomorphisms
\begin{align*}
    &\Delta_r^{\G, \mathbb{H}}: B(L^2(\G))\to B(L^2(\G))\bar{\otimes}L^\infty(\mathbb{H}): x \mapsto V_{\G, \mathbb{H}}(x\otimes 1)V_{\G, \mathbb{H}}^*\\
    &\Delta_l^{\mathbb{H}, \G}: B(L^2(\G))\to L^\infty(\mathbb{H})\ovot B(L^2(\G)): x \mapsto W_{\mathbb{H}, \G}^*(1\otimes x)W_{\mathbb{H}, \G}.
\end{align*}

 The map $\Delta_r^{\G, \mathbb{H}}$ defines an action $B(L^2(\G))\curvearrowleft \mathbb{H}$ and the map $\Delta_l^{\mathbb{H}, \G}$ defines an action $\mathbb{H}\curvearrowright B(L^2(\G))$. These maps restrict to coactions
 $$\Delta^{\G, \mathbb{H}}: L^\infty(\G)\to L^\infty(\G)\ovot L^\infty(\mathbb{H}), \quad \Delta^{\mathbb{H}, \G}: L^\infty(\G)\to L^\infty(\mathbb{H}) \ovot L^\infty(\G).$$
 Let us also recall (see e.g.\ \cite{DC09}*{Definition 6.5.3}) that if $(A, \alpha)$ is a $\G$-$W^*$-algebra, then there is a unique $\mathbb{H}$-action $\alpha_\mathbb{H}: A \to A \ovot L^\infty(\mathbb{H})$ such that 
 $$(\alpha \otimes \id)\circ \alpha_{\mathbb{H}}= (\id \otimes \Delta^{\G, \mathbb{H}})\circ\alpha.$$
We refer to the action $\alpha_{\mathbb{H}}: A \curvearrowleft \mathbb{H}$ as the \emph{restriction} of the action $\alpha: A \curvearrowleft \G$ to the closed quantum subgroup $\mathbb{H}$.
\begin{Prop}
    Let $(A, \alpha)$ and $(B, \beta)$ be $\G$-$W^*$-algebras and let $\mathbb{H}$ a closed quantum subgroup of $\G$. If $(A, \alpha)\sim_\G (B,\beta)$, then also $(A, \alpha_{\mathbb{H}})\sim_{\mathbb{H}} (B, \beta_{\mathbb{H}})$.
\end{Prop}
\begin{proof}
Let $(\mathcal{H}, \pi, \rho, U)\in \Corr^\G(A,B)$ be a $\G$-Morita correspondence. Consider the restriction $U_\mathbb{H}\in B(\mathcal{H})\ovot L^\infty(\mathbb{H})$ (see e.g.\ \cite{DC09}*{Section 6.5}) uniquely determined by
$$(\id \otimes \Delta^{\G, \mathbb{H}})(U) = U_{12} U_{\mathbb{H},13}, \quad (\id \otimes \Delta^{\mathbb{H}, \G})(U) = U_{\mathbb{H},12} U_{13}.$$ 
Using the identity
$$\Delta^{\mathbb{H}, \G}(R_\G(x))= (R_\mathbb{H}\otimes R_\G) (\Delta^{\G, \mathbb{H}}(x)_{21}), \quad x \in L^\infty(\G),$$
routine verifications show that $(\mathcal{H}, \pi, \rho, U_\mathbb{H}) \in \Corr^\mathbb{H}(A,B)$ defines a $\mathbb{H}$-Morita correspondence. 
\end{proof}

 Fix now a $\mathbb{H}$-$W^*$-algebra $(A, \alpha)$. Consider the canonical $\check{\G}$-action on the crossed product $A\rtimes_\alpha \mathbb{H}$ given by
 $$\delta:= (\id \otimes \id \otimes \check{\gamma})\circ (\id \otimes \check{\Delta}_r^{\mathbb{H}}): A\rtimes_\alpha \mathbb{H}\to (A\rtimes_\alpha \mathbb{H})\ovot L^\infty(\check{\G}).$$

 We also define the cotensor product 
   $$A \overset{\mathbb{H}}{\square} L^\infty(\G):= \{z \in A \ovot L^\infty(\G): (\alpha \otimes \id)(z) = (\id \otimes \Delta^{\mathbb{H}, \G})(z)\}.$$

Slicing with functionals in $L^1(\G)$ and using that $(\id \otimes \Delta^\G)\Delta^{\mathbb{H}, \G}= (\Delta^{\mathbb{H}, \G}\otimes \id)\Delta^\G$ shows that $$(\id \otimes \Delta^\G)(A \overset{\mathbb{H}}{\square} L^\infty(\G))\subseteq (A \overset{\mathbb{H}}{\square} L^\infty(\G))\ovot L^\infty(\G),$$ so that $A \overset{\mathbb{H}}{\square} L^\infty(\G)$ becomes a $\G$-$W^*$-algebra. We refer to $(A \overset{\mathbb{H}}{\square} L^\infty(\G), \id \otimes \Delta^\G)$ as the \emph{induction} of $(A, \alpha)$.

The following is contained in \cite{Vae05}*{Section 7} in the $C^*$-algebra context.
 \begin{Exa}\label{example closed}
     $(A\rtimes_{\alpha}\mathbb{H}, \delta)\sim_{\check{\G}} ((A \overset{\mathbb{H}}{\square} L^\infty(\G)) \rtimes_{\id \otimes \Delta^\G} \G, \id \otimes \id \otimes \check{\Delta}_r^\G).$

To prove this, we can simply provide an explicit $\check{\G}$-Morita correspondence on the Hilbert space $L^2(A)\otimes L^2(\G)$, namely, consider the following data:
\begin{itemize}\setlength\itemsep{-0.5em}
    \item The $*$-representation $\pi= \pi_A\otimes (\Delta_l^{\G})^{-1}: (A \overset{\mathbb{H}}{\square} L^\infty(\G))\rtimes_{\id \otimes \Delta^\G}\G \to B(L^2(A)\otimes L^2(\G)),$
    \item the anti-$*$-representation $\rho: A\rtimes_\alpha \mathbb{H}\to B(L^2(A)\otimes L^2(\G)): z \mapsto (\id \otimes \hat{\gamma})(U_\alpha(\rho_A\otimes \check{J}_{\mathbb{H}}(-)^*\check{J}_{\mathbb{H}})(z)U_\alpha^*),$
    \item the unitary $\check{\G}$-representation $\check{V}_{\G,23}\in B(L^2(A)\otimes L^2(\G))\ovot L^\infty(\check{\G})$.
\end{itemize}

The only non-obvious thing to check in the definition of $\check{\G}$-Morita correspondence is that the images of $\pi$ and $\rho$ are each other's commutant. To prove this, we can make use of the argument implicitly provided in \cite{Vae05}*{Section 7}, which we now sketch. Put
  $$X:= \{x \in A \ovot B(L^2(\mathbb{H}), L^2(\G)): (\alpha\otimes \id)(x) = W_{\mathbb{H}, \G, 23}^*x_{13}W_{\mathbb{H}, 23}\}.$$
 It can then be checked that
 $$(\id \otimes \Delta_l^\G)([XX^*]^{\sigma\text{-weak}})= (A \overset{\mathbb{H}}{\square} L^\infty(\G))\rtimes \G, \quad [X^*X]^{\sigma\text{-weak}}= A\rtimes_\alpha \mathbb{H}.$$
 Therefore, $X$ becomes an equivalence $(A \overset{\mathbb{H}}{\square} L^\infty(\G))\rtimes \G$-$A\rtimes_\alpha \mathbb{H}$-bimodule in a natural way. It is easy to see that the map
 $$X \otimes_{A\rtimes_\alpha \mathbb{H}} (L^2(A)\otimes L^2(\mathbb{H}))\to L^2(A)\otimes L^2(\G): x \otimes_{A\rtimes_\alpha \mathbb{H}} (\xi \otimes \eta)\mapsto (\pi_A\otimes \id)(x)(\xi \otimes \eta)$$
 is a unitary isomorphism of $(A \overset{\mathbb{H}}{\square} L^\infty(\G))\rtimes \G$-$A\rtimes_\alpha \mathbb{H}$-correspondences. Proposition \ref{technical} then implies that the images of $\pi$ and $\rho$ are each other's commutant.%\footnote{In fact, with the coaction $\gamma: X \to X \ovot L^\infty(\check{\G}): x \mapsto  \check{V}_{\G, 23}x_{12} \check{V}_{\mathbb{H}, \G, 23}^*$, the pair $(X, \gamma)$ becomes an equivalence $\check{\G}$-$(A \overset{\mathbb{H}}{\square} L^\infty(\G))\rtimes \G$-$A\rtimes_\alpha \mathbb{H}$-bimodule, cfr.\ Remark \ref{a}.}
 \end{Exa}

 Applying example \ref{example closed} to $A= \mathbb{C}$, we find that 
 \begin{equation}\label{particular}
    L^\infty(\check{\mathbb{H}})\sim_{\check{\G}} L^\infty(\mathbb{H}\backslash \G) \rtimes_{\Delta^\G} \G,
 \end{equation}
 where $L^\infty(\mathbb{H}\backslash \G):= \{x\in L^\infty(\G): \Delta^{\mathbb{H}, \G}(x)=  1\otimes x\}.$
It is easily checked that we have $\Delta^\G(L^\infty(\mathbb{H}\backslash \G))\subseteq L^\infty(\mathbb{H}\backslash \G)\ovot L^\infty(\G)$, i.e.\ $L^\infty(\mathbb{H}\backslash \G)$ is a right coideal von Neumann algebra. It turns out that the $\check{\G}$-$W^*$-Morita equivalence \eqref{particular} generalizes to arbitrary coideal von Neumann algebras, as we now explain.

Let $L^\infty(\mathbb{K}\backslash \G) \subseteq L^\infty(\G)$ be a right coideal von Neumann algebra, i.e.\ $L^\infty(\mathbb{K}\backslash \G)$ is a von Neumann subalgebra of $L^\infty(\G)$ and $\Delta(L^\infty(\mathbb{K}\backslash \G))\subseteq L^\infty(\mathbb{K}\backslash \G) \ovot L^\infty(\G).$ We define 
$$L^\infty(\hat{\mathbb{K}}):= L^\infty(\mathbb{K}\backslash \G)'\cap L^\infty(\hat{\G})\subseteq L^\infty(\hat{\mathbb{G}}), \quad L^\infty(\check{\mathbb{K}}) := \check{J} L^\infty(\hat{\mathbb{K}}) \check{J}\subseteq L^\infty(\check{\mathbb{G}}).$$
One can show that (cfr.\ \cite{KS14}*{Proposition 3.5})
$$\hat{\Delta}(L^\infty(\hat{\mathbb{K}})) \subseteq L^\infty(\hat{\mathbb{K}}) \ovot L^\infty(\hat{\G}), \quad \check{\Delta}(L^\infty(\check{\mathbb{K}})) \subseteq L^\infty(\check{\mathbb{K}}) \ovot L^\infty(\check{\G}),$$
and we refer to the von Neumann algebras $L^\infty(\check{\mathbb{K}})$ and $L^\infty(\hat{\mathbb{K}})$ as \emph{dual coideal von Neumann algebras}. The notations used here are of course inspired from the case where $\mathbb{K}= \mathbb{H}$ is a closed quantum subgroup of $\G$.

In the non-equivariant setting, it is well-known that coideal von Neumann algebras give rise to a specific $W^*$-Morita equivalence, see e.g.\ \cites{ILS98, KS14,AV23}. We can upgrade this $W^*$-Morita equivalence to the equivariant setting:

\begin{Prop}\label{coideal Morita}
    $(L^\infty(\mathbb{K}\backslash \G)\rtimes_\Delta \G, \id \otimes \check{\Delta}_r)\sim_{\check{\G}} (L^\infty(\check{\mathbb{K}}), \check{\Delta})$.
\end{Prop}
\begin{proof} We verify that $L^2(\G)$ becomes a $\check{\G}$-$L^\infty(\mathbb{K}\backslash \G)\rtimes_\Delta \G$-$L^\infty(\check{\mathbb{K}})$-Morita correspondence for the following data:
    \begin{itemize}\setlength\itemsep{-0.5em}
        \item The faithful $*$-representation $\pi= \Delta_l^{-1}: L^\infty(\mathbb{K}\backslash \G)\rtimes_\Delta \G \to B(L^2(\G))$ uniquely determined by
        $$\pi(\Delta(x))= x, \quad \pi(1\otimes \check{y})= \check{y}, \quad x \in L^\infty(\mathbb{K}\backslash \G), \quad \check{y}\in L^\infty(\check{\G}),$$
        \item the faithful anti-$*$-representation $\rho: L^\infty(\check{\mathbb{K}}) \to B(L^2(\G)): \check{y} \mapsto \check{J}\check{y}^* \check{J}$,
        \item the unitary $\check{V}\in B(L^2(\G))\bar{\otimes} L^\infty(\check{\G})$.
    \end{itemize}
    We calculate for $x\in L^\infty(\G)$ and $\check{y}\in L^\infty(\check{\G})$ that
    \begin{align*}
        &(\pi\otimes \id)(\id \otimes \check{\Delta}_r)(\Delta(x)) = (\pi\otimes \id)(\Delta(x)\otimes 1) = x\otimes 1 = \check{V}(x\otimes 1)\check{V}^*= \check{V}(\pi(\Delta(x))\otimes 1)\check{V}^*,\\
        & (\pi\otimes \id)(\id \otimes \check{\Delta}_r)(1\otimes \check{y})= (\pi\otimes \id)(1\otimes \check{\Delta}(\check{y}))= \check{\Delta}(\check{y}) = \check{V}(\pi(1\otimes \check{y}) \otimes 1)\check{V}^*,
    \end{align*}
    so we conclude that $(\pi\otimes \id)(\id \otimes \check{\Delta}_r)(z) = \check{V}(\pi(z)\otimes 1)\check{V}^*$ for all $z\in L^\infty(\mathbb{K}\backslash\G)\rtimes_\Delta \G$. On the other hand, if $\check{y}\in L^\infty(\check{\mathbb{K}})$, we have
    \begin{align*}
       \check{V}^*(\rho(\check{y})\otimes 1)\check{V}&= \check{V}^*(\check{J}\check{y}^*\check{J}\otimes 1)\check{V}=(\check{J}\otimes J)\check{V}(\check{y}^*\otimes 1)\check{V}^*(\check{J}\otimes J)= (\check{J}\otimes J)\check{\Delta}(\check{y}^*)(\check{J}\otimes J)= (\rho \otimes \check{R})(\check{\Delta}(\check{y})).
    \end{align*}
    Finally, note that $\pi(L^\infty(\mathbb{K}\backslash \G)\rtimes_\Delta \G)' = L^\infty(\hat{\mathbb{K}}) = \check{J}L^\infty(\check{\mathbb{K}}) \check{J}=  \rho(L^\infty(\check{\mathbb{K}})).$
\end{proof}

\section{Preservation of dynamical properties under equivariant Morita equivalence}

In this section, we prove that equivariant Morita equivalence preserves (strong) inner amenability, (strong) amenability and equivariant injectivity of dynamical systems. We apply this to the natural Morita equivalence associated to a coideal von Neumann subalgebra $L^\infty(\mathbb{K}\backslash \G)\subseteq L^\infty(\G)$ and use this to relate approximation properties of $L^\infty(\mathbb{K}\backslash \G)$ and $L^\infty(\check{\mathbb{K}})$. We use this to resolve some open questions in the literature, illustrating the use of equivariant Morita equivalence in the theory of locally compact quantum groups.

Let us recall from \cite{DCDR24a}*{Definition 4.1} that a $\G$-equivariant unital inclusion $A \subseteq B$ is called \begin{itemize}\setlength\itemsep{-0.5em}
    \item \emph{strongly $\G$-amenable}  if we have the $\G$-equivariant weak containment ${}_A L^2(A)_A\preccurlyeq {}_A L^2(B)_A$ of $\G$-$A$-$A$-correspondences,
    \item \emph{$\G$-amenable} if there exists a $\G$-equivariant ucp conditional expectation $E: B \to A$.
\end{itemize}
In particular, we say that $(A, \alpha)$ is \emph{(strongly) $\G$-amenable} if the $\G$-inclusion $\alpha(A)\subseteq A\ovot L^\infty(\G)$ is (strongly) amenable and we say that $(A, \alpha)$ is (strongly) \emph{inner $\G$-amenable} if the $\G$-inclusion $\alpha(A)\subseteq A\rtimes_\alpha \G$ (with $A\rtimes_\alpha \G$ endowed with the adjoint $\G$-action $\id \otimes \Delta_r: A\rtimes_\alpha \G \to (A\rtimes_\alpha \G)\ovot L^\infty(\G)$) is (strongly) amenable.   From \cite{DCDR24a}*{Definition 6.6}, we recall that $(A, \alpha)$ is strongly $\G$-amenable if and only if the trivial $\G$-$A$-$A$ correspondence is $\G$-weakly contained in the $\G$-semi-coarse $\G$-$A$-$A$-correspondence $S_A^\G$  and that $(A, \alpha)$ is strongly inner $\G$-amenable if and only if the trivial $\G$-$A$-$A$-correspondence is $\G$-weakly contained in the adjoint $\G$-$A$-$A$-correspondence $D_A^\G$. For the definition of \emph{$\G$-injectivity} of $(A, \alpha)$, we refer the reader to \cite{DR24a}, but we mention that $(A, \alpha)$ is $\G$-injective if and only if $A$ is injective and $(A,\alpha)$ is $\G$-amenable \cite{DR24a}*{Proposition 3.10}.

\begin{Theorem}\label{MoritaApplication}
    Let $(A, \alpha)$ and $(B, \beta)$ be $\G$-$W^*$-algebras such that $(A, \alpha)\sim_\G (B, \beta)$. The following statements hold:
    \begin{enumerate}\setlength\itemsep{-0.5em}
        \item $(A, \alpha)$ is (strongly) $\G$-amenable if and only if $(B, \beta)$ is (strongly) $\G$-amenable.
\item $(A, \alpha)$ is (strongly) inner $\G$-amenable if and only if $(B, \beta)$ is (strongly) inner $\G$-amenable.
\item $(A, \alpha)$ is $\G$-injective if and only if $(B, \beta)$ is $\G$-injective.
    \end{enumerate}
\end{Theorem}
\begin{proof}
    $(1)$ It suffices to show that if $(B,\beta)$ is (strongly) $\G$-amenable, then so is $(A, \alpha)$. 

  In the strong case, choose a $\G$-Morita-correspondence $\mathcal{H}\in \Corr^\G(A,B)$. By assumption, we have $L^2(B) \preccurlyeq_\G S_B^\G$. But then, using \cite{DCDR24a}*{Proposition 5.10, Example 5.14, Example 5.20, Proposition 6.3}, we find
  $$L^2(A) \cong \mathcal{H}\boxtimes_B L^2(B) \boxtimes_B \overline{\mathcal{H}}\preccurlyeq_\G \mathcal{H}\boxtimes_B S_B^\G \boxtimes \overline{\mathcal{H}}\cong S_A^\G \boxtimes_A\mathcal{H}\boxtimes_B \overline{\mathcal{H}}\cong S_A^\G,$$
    so that $(A, \alpha)$ is also strongly $\G$-amenable.

In the non-strong case, we note that $(B, \beta)$ is $\G$-amenable if and only if $(B(\mathcal{K})\ovot B, \id \otimes \beta)$ is $\G$-amenable and that $\G$-amenability is preserved under $\G$-equivariant isomorphism. Therefore, by Theorem \ref{char2}, we may assume that $(A, \alpha)= (pBp, \beta_P)$ (cfr.\ Example \ref{main example}), where $p\in B$ is a (full) projection and $P\in B \ovot L^\infty(\G)$ satisfies
    $$(\id \otimes \Delta)(P) = (P\otimes 1) (\beta \otimes \id)(P), \quad P^*P=\beta(p), \quad PP^* = p\otimes 1.$$

By assumption, there exists a $\G$-equivariant ucp map
$$E_B: (B \ovot L^\infty(\G), \id \otimes \Delta)\to (B, \beta)$$
    such that $E_B\circ \beta = \id_B$. We then define the ucp map
    $$E_A: (pBp) \ovot L^\infty(\G)\to pBp: z \mapsto pE_B(P^*zP)p.$$

    If $a \in pBp$, then $E_A(\beta_P(a)) = pE_B(P^*P \beta(a)P^*P)p= pE_B(\beta(a))p = pap = a$, and if $z\in (pBp)\ovot L^\infty(\G)$, then
    \begin{align*}
    \beta_P(E_A(z))&= P P^*P \beta(E_B(P^*zP))P^*PP^*\\
    &= P (E_B\otimes \id)(\id \otimes \Delta)(P^*zP) P^*\\
    &= P(E_B\otimes \id)((\beta \otimes \id)(P^*)P_{12}^*(\id \otimes \Delta)(z)P_{12}(\beta \otimes \id)(P))P^*\\
    &=PP^* (E_B\otimes \id)(P_{12}^* (\id \otimes \Delta)(z) P_{12}) PP^*\\
    &= (p\otimes 1)(E_B\otimes \id)(P_{12}^*(\id \otimes \Delta)(z)P_{12})(p\otimes 1)\\
    &= (E_A\otimes \id)(\id \otimes \Delta)(z),
    \end{align*}
    where the fourth equality follows from a multiplicative domain argument. It follows that $$E_A: (A \ovot L^\infty(\G), \id \otimes \Delta)\to (A, \beta_P)$$ is a $\G$-equivariant ucp map satisfying $E_A\circ \beta_P = \id$, so that $(A, \alpha)$ is also $\G$-amenable.

    (2) The same proof strategy as in the proof of $(1)$ works. Alternatively, this follows from the following equivalences:
    \begin{align*}
(A, \alpha) \textrm{\ is (strongly) inner } \G \textrm{-amenable}&\iff  (A\rtimes_\alpha \G, \id \otimes \check{\Delta}_r) \textrm{\ is (strongly) } \check{\G} \textrm{-amenable}\\
&\iff (B\rtimes_\beta \G, \id \otimes \check{\Delta}_r) \textrm{\ is (strongly) } \check{\G} \textrm{-amenable}\\
&\iff(B, \beta) \textrm{\ is (strongly) inner } \G \textrm{-amenable}.
    \end{align*}
Here, the first and third equivalence uses \cite{DCDR24a}*{Theorem 6.12} and the second equivalence combines $(1)$ and Proposition \ref{crossed product}.

(3) This follows since both $\G$-amenability and injectivity are preserved under $\G$-equivariant Morita equivalence.
\end{proof}

\begin{Cor}\label{app1} Let $\mathbb{H}$ be a Vaes-closed quantum subgroup of the locally compact quantum group $\G$. The following statements hold for a $\mathbb{H}$-$W^*$-algebra $(A, \alpha)$:
\begin{enumerate}\setlength\itemsep{-0.5em}
    \item $A \overset{\mathbb{H}}{\square} L^\infty(\G)$ is (strongly) $\G$-amenable if and only if $A\rtimes_\alpha \mathbb{H}$ is (strongly) inner $\check{\G}$-amenable.
    \item $A \overset{\mathbb{H}}{\square} L^\infty(\G)$ is (strongly) inner $\G$-amenable if and only if $A\rtimes_\alpha \mathbb{H}$ is (strongly) $\check{\G}$-amenable.
    \item $A \overset{\mathbb{H}}{\square} L^\infty(\G)$ is $\G$-injective if and only if $A\rtimes_\alpha \mathbb{H}$ is $\check{\G}$-injective.
\end{enumerate}
\end{Cor}
\begin{proof}
    Consider the $\check{\G}$-$W^*$-Morita equivalence of Example \ref{example closed}. The equivalences $(1)$ and $(2)$ then follow by combining Theorem \ref{MoritaApplication} and \cite{DCDR24a}*{Theorem 6.12}. The equivalence $(3)$ follows by combining Theorem \ref{MoritaApplication} and \cite{DR24a}*{Corollary 5.3}.
\end{proof}

\begin{Rem}
    In \cite{Cr17b}*{Theorem 3.5}, it was shown that for a locally compact quantum group $\G$ with $\check{\G}$ amenable and for a Vaes-closed quantum subgroup $\mathbb{H}$ of $\G$, it holds that $L^\infty(\check{\mathbb{H}})$ is $\check{\G}$-injective if and only if $\mathbb{H}$ is amenable. Using Corollary \ref{app1}, we can give an alternative, shorter proof of the backwards implication. Indeed, if $m_\mathbb{H}: (L^\infty(\mathbb{H}), \Delta_\mathbb{H})\to (\C, \tau)$ is a $\mathbb{H}$-equivariant state, then
    $$(m_\mathbb{H}\otimes \id) \circ \Delta^{\mathbb{H}, \G}: (L^\infty(\G), \Delta^\G)\to (L^\infty(\mathbb{H}\backslash \G), \Delta^\G)$$
    is a $\G$-equivariant ucp conditional expectation. Since $\check{\G}$ is amenable, we know that $(L^\infty(\G), \Delta^\G)$ is $\G$-injective \cite{Cr17a}*{Theorem 5.1} and thus also $(L^\infty(\mathbb{H}\backslash \G), \Delta^\G)$ is $\G$-injective. Applying Corollary \ref{app1} (with $A= \mathbb{C}$), we deduce that $L^\infty(\check{\mathbb{H}})$ is $\check{\G}$-injective.
\end{Rem}

In exactly the same way, by combining Theorem \ref{MoritaApplication} and the $\check{\G}$-$W^*$-Morita equivalence from Proposition \ref{coideal Morita}, we find:

\begin{Cor}\label{app2}
    Let $L^\infty(\mathbb{K}\backslash \G)\subseteq L^\infty(\G)$ be a (right) coideal von Neumann algebra. The following statements hold:
    \begin{enumerate}\setlength\itemsep{-0.5em}
        \item $L^\infty(\mathbb{K}\backslash \G)$ is (strongly) $\G$-amenable if and only if $L^\infty(\check{\mathbb{K}})$ is (strongly) inner $\check{\G}$-amenable.
        \item $L^\infty(\mathbb{K}\backslash \G)$ is (strongly) inner $\G$-amenable if and only if $L^\infty(\check{\mathbb{K}})$ is (strongly) $\check{\G}$-amenable.
        \item $L^\infty(\mathbb{K}\backslash \G)$ is $\G$-injective if and only if $L^\infty(\check{\mathbb{K}})$ is $\check{\G}$-injective.
    \end{enumerate}
\end{Cor}

We can then easily prove the following, part of which refines the important result \cite{AK24}*{Theorem 5.9}. For the definition of compact quasi-subgroup $\mathbb{K}$ of a compact quantum group $\G$, we refer the reader to \cite{AK24}*{Section 2.4}, but we mention that a closed quantum subgroup $\mathbb{H}$ of a compact quantum group is of this type.
\begin{Theorem}\label{main}
Let $\G$ be a compact quantum group and $L^\infty(\mathbb{K}\backslash \G)$ a (right) coideal von Neumann algebra. The following statements are equivalent:
\begin{enumerate}\setlength\itemsep{-0.5em}
    \item $(\ell^\infty(\check{\mathbb{K}}), \check{\Delta})$ is $\check{\G}$-amenable.
    \item $(\ell^\infty(\check{\mathbb{K}}), \check{\Delta})$ is $\check{\G}$-injective.
    \item $(L^\infty(\mathbb{K}\backslash \G), \Delta)$ is inner $\G$-amenable.
    \item $(L^\infty(\mathbb{K}\backslash \G), \Delta)$ is $\G$-injective.
    \item $\mathbb{K}$ is a compact quasi-subgroup of $\G$ and $\ell^\infty(\check{\mathbb{K}})$ is relatively amenable in $\ell^\infty(\check{\G})$, i.e.\ there exists a $\check{\G}$-equivariant ucp map $\ell^\infty(\check{\G})\to \ell^\infty(\check{\mathbb{K}}).$
\end{enumerate}
\end{Theorem}

\begin{proof}
    Since $\ell^\infty(\check{\mathbb{K}})$ is isomorphic to a direct product of matrix algebras, it is an injective von Neumann algebra. Consequently, the equivalence $(1)\iff (2)$ follows from \cite{DR24a}*{Proposition 3.10}. The equivalences $(1) \iff (3)$ and $(2) \iff (4)$ are proven in Corollary \ref{app2}. The implication $(5)\implies (4)$ is proven in \cite{AK24}*{Theorem 5.9}. Finally, if $(2)$ holds, then clearly $\ell^\infty(\check{\mathbb{K}})$ is relatively amenable in $\ell^\infty(\check{\G})$. Since $(2)\implies (4)$, we know from \cite{DH24}*{Lemma 3.17} that $C(\mathbb{K}\backslash \G)= \mathcal{R}(L^\infty(\mathbb{K}\backslash \G))$ is $\G$-$C^*$-injective. Then \cite{AK24}*{Proposition 5.8} shows that $\mathbb{K}$ is a compact quasi-subgroup of $\G$. We thus have shown that $(2)\implies (5)$.
    \end{proof}

\begin{Rem} Let us remark that if $\mathbb{G}$ is a co-amenable compact quantum group, then $\ell^\infty(\check{\mathbb{K}})$ is automatically relatively amenable in $\ell^\infty(\check{\G})$. Theorem \ref{main} provides us with a rich class of examples of non-amenable actions of (amenable) discrete quantum groups, cfr.\ \cite{DCDR24b}*{Section 4}. \end{Rem}

We now give several applications of this result. We start with the following:
\begin{Prop}
    Let $\G$ be a compact quantum group. The following are equivalent:
    \begin{enumerate}\setlength\itemsep{-0.5em}
        \item The natural action $\alpha: L^\infty(\G)\curvearrowleft \G^{\op}\times \G: x \mapsto \Delta_\G^{(2)}(x)_{213}$ is inner amenable.
        \item $(L^\infty(\G), \alpha)$ is $\G^{\op}\times \G$-injective.
        \item $\G$ is of Kac type and $\G$ is co-amenable.
    \end{enumerate}
\end{Prop}
\begin{proof} We start by noting that $B:=\Delta_\G(L^\infty(\G))$ satisfies $\Delta_{\G^{\op}\times \G}(B)\subseteq B \ovot L^\infty(\G^{\op}\times \G)$ and that 
$$\Delta_\G: (L^\infty(\G), \alpha)\to (B, \Delta_{\G^{\op}\times \G})$$
is a $\G$-equivariant $*$-isomorphism. Therefore, the result follows from the following equivalences:
\begin{align*}\alpha: L^\infty(\G)\curvearrowleft \G^{\op}\times \G \text{ is inner amenable} &\iff (L^\infty(\G), \alpha) \text{ is } \G^{\op}\times \G\text{-injective}\\
&\iff L^\infty(\G) \text{ is injective and } \alpha: L^\infty(\G)\curvearrowleft \G^{\op}\times \G \text{ is amenable}\\
&\iff L^\infty(\G) \text{ is injective and } \G \text{ is of Kac type}\\
&\iff \G \text{ is of Kac type and } \G \text{ is co-amenable.}
\end{align*}
The first equivalence is Theorem \ref{main}, the second equivalence is \cite{DR24a}*{Proposition 3.10}, the third equivalence is \cite{DCDR24b}*{Theorem 4.4} and the last equivalence is well-known (see e.g.\ \cite{Rua96}).
\end{proof}

Using Theorem \ref{main} and the results in \cite{AK24}, we can answer the open question \cite{KKSV22}*{Question 8.1}. Note that the authors had established this result to be true for a normal quantum subgroup $\mathbb{H}$ \cite{KKSV22}*{Theorem 3.10}.
\begin{Cor}\label{ezcor}
    Let $\mathbb{H}$ be a closed quantum subgroup of the compact quantum group $\G$. Then $\mathbb{H}\backslash \G$ is co-amenable if and only if $\ell^\infty(\check{\mathbb{H}})$ is relative amenable in $\ell^\infty(\check{\G})$.
\end{Cor}
\begin{proof}
    By \cite{AK24}*{Theorem 3.2 + Theorem 6.5} it is known that $\mathbb{H}\backslash \G$ is co-amenable if and only if $\ell^\infty(\check{\mathbb{H}})$ is $\check{\G}$-injective, which by Theorem \ref{main} is equivalent with relative amenability of $\ell^\infty(\check{\mathbb{H}})$ in $\ell^\infty(\check{\G}).$
\end{proof}

In \cite{KKSV22}, the notion of \emph{Furstenberg boundary} $\partial_F \mathbbl{\Gamma}$ for a discrete quantum group $\mathbbl{\Gamma}$ was introduced. The function algebra $C(\partial_F \mathbbl{\Gamma})$ carries a left $C^*$-algebraic coaction $\alpha_F: C(\partial_F\mathbbl{\Gamma})\to M(c_0(\mathbbl{\Gamma})\otimes C(\partial_F\mathbbl{\Gamma}))$. We call the action $\alpha_F: \mathbbl{\Gamma}\curvearrowright C(\partial_F \mathbbl{\Gamma})$ \emph{faithful} if $$\{(\id \otimes \nu)\alpha(x): x \in C(\partial_F\mathbbl{\Gamma}), \nu \in C(\partial_F \mathbbl{\Gamma})^*\}''= \ell^\infty(\mathbbl{\Gamma}).$$
If $\mathbbl{\Gamma}$ is a classical discrete group, this coincides with the usual notion of faithfulness of an action.
By analogy with the theory of discrete groups \cite{BKKO17}, we expect that simplicity of the $C^*$-algebra $C(\hat{\mathbbl{\Gamma}})$ implies faithfulness of the action $\alpha_F: \mathbbl{\Gamma}\curvearrowright C(\partial_F\mathbbl{\Gamma})$.\footnote{The converse is known to be false, as there exist discrete groups that satisfy the unique trace property, but which are not $C^*$-simple.} In \cite{AK24}*{Theorem 6.9}, this was established in the case that $\mathbbl{\Gamma}$ is either exact or unimodular. Inspecting the proof of this result and using Theorem \ref{main}, we can dispose of these assumptions on the discrete quantum group $\mathbbl{\Gamma}$:
\begin{Prop}
    Let $\mathbbl{\Gamma}$ be a discrete quantum group. If $C(\hat{\mathbbl{\Gamma}})$ is a simple $C^*$-algebra, then the action $\alpha_F: \mathbbl{\Gamma}\curvearrowright C(\partial_F \mathbbl{\Gamma})$ is faithful.
\end{Prop}

\section{Comparison with the $C^*$-setting for compact quantum groups}

Throughout this entire section, we fix a (reduced) compact quantum group $\G$ with Haar state $\Phi: L^\infty(\G)\to \mathbb{C}$. Let us start by recalling some definitions from the theory of actions of compact quantum groups. For details, we refer the reader to the lecture notes \cite{DC17}. 

\begin{itemize}
\item By a \emph{unital $\G$-$C^*$-algebra} $(X, \alpha)$, we mean a unital $C^*$-algebra $X$ together with a unital isometric $*$-homomorphism
$\alpha: X \to X \otimes C(\G)$
such that $(\alpha\otimes \id)\alpha = (\id \otimes \Delta)\alpha$ and $[\alpha(X)(1\otimes C(\G))]= X\otimes C(\G)$.
    \item If $(A, \alpha)$ is a $\G$-$W^*$-algebra, then we define
\begin{align*}
    \mathcal{R}_{\operatorname{alg}}(A)&:= \{a\in A: \alpha(a)\in A \odot \mathcal{O}(\G)\},\\
    \mathcal{R}(A)&:= \overline{\mathcal{R}_{\operatorname{alg}}(A)}^{\|\cdot\|}=[(\id \otimes \omega)\alpha(a): a \in A, \omega \in L^1(\G)].
\end{align*}
We have $\alpha(\mathcal{R}(A))\subseteq \mathcal{R}(A)\otimes C(\G)$ and $(\mathcal{R}(A), \alpha)$ becomes a unital $\G$-$C^*$-algebra. We refer to $\mathcal{R}(A)$ as the \emph{regular elements} of $A$ (with respect to the action $\alpha: A \curvearrowleft \G$). It is well-known that $\mathcal{R}(A)$ is $\sigma$-weakly dense in $A$.
%\item If $(X, \alpha)$ and $(Y, \beta)$ are unital $\G$-$C^*$-algebras, a \emph{linking $\G$-$C^*$-algebra} between $(X, \alpha)$ and $(Y, \beta)$ consists of the data $(Z, \gamma, p, \mu, \nu)$ where $(Z, \gamma)$ is a unital $\G$-$C^*$-algebra, $p\in Z^\gamma$ is a projection satisfying $[ZpZ]= Z$ and $[Zp^\perp Z]= Z$ and $\mu: pZp \to X$ and $\nu: p^\perp Z p^\perp \to B$ are $\G$-equivariant isomorphisms of $C^*$-algebras.
\item  Given $\G$-$C^*$-algebras $(X, \alpha)$ and $(Y, \beta)$, an \emph{equivalence $\G$-$X$-$Y$-bimodule} \cite{NV10}*{Section 2} consists of a full $\G$-equivariant right Hilbert $Y$-module $E$ together with a $\G$-equivariant $*$-isomorphism $X \cong \mathcal{K}_Y(E)$. We call $(X, \alpha)$ and $(Y, \beta)$ \emph{$\G$-$C^*$-Morita equivalent} if there exists an equivalence $\G$-$X$-$Y$-bimodule.
\end{itemize}

Consider now the following two statements:
\begin{enumerate}\setlength\itemsep{-0.5em}
    \item $(A, \alpha)$ and $(B, \beta)$ are $\G$-$W^*$-Morita equivalent.
    \item $(\mathcal{R}(A), \alpha)$ and $(\mathcal{R}(B), \beta)$ are $\G$-$C^*$-Morita equivalent.
\end{enumerate}
It is natural to ask if there is a relation between the statements $(1)$ and $(2)$. Obviously (take $\G$ trivial), the implication $(1)\implies (2)$ is not true in general. The main goal of this section is to show that $(2)\implies (1)$ and that $(1)\iff (2)$ in the case that $(A, \alpha)$ and $(B, \beta)$ are both ergodic. In order to establish these results, we need the following result, which is of independent interest:

\begin{Theorem}\label{important inbetween}
Let $(A, \alpha)$ and $(B, \beta)$ be $\G$-$W^*$-algebras. Every $\G$-equivariant $*$-isomorphism $\phi: \mathcal{R}(A)\to \mathcal{R}(B)$ extends (uniquely) to a ($\G$-equivariant) $*$-isomorphism $A \to B$. 
\end{Theorem}
We will now prove this result. For this, we will make use of the theory of Hilbert modules. Let us recall some facts/notations from \cite{Pa73} that we will be using.

Let $B$ be a $C^*$-algebra and $X$ be a right pre-Hilbert $B$-module (i.e.\ $X$ satisfies all the data in the definition of a Hilbert $B$-module, except completeness of $X$ with respect to the norm induced by the $B$-valued inner product on $X$). One can still make sense of the $*$-algebra of adjointable operators $$\mathscr{L}_B(X):=\{t \in B(X)\mid \exists t^*\in B(X): \forall x,y \in X: \langle x,ty\rangle = \langle t^*x, y\rangle\}.$$ If $X$ is complete, then so is $\mathscr{L}_B(X)$.

We call $X$ \emph{self-dual} \cite{Pa73} if for every bounded $B$-linear map $T:X\to B$, there exists $x\in X$ such that $T(y)= \langle x,y \rangle_B$ for all $y\in X$. If $X$ is self-dual, then it is necessarily complete. Moreover, if $X$ is self-dual and $B$ is a von Neumann algebra, then the space $\mathscr{L}_B(X)$ carries the canonical structure of a $W^*$-algebra \cite{Pa73}*{Proposition 3.10}. Indeed, given $t\in \mathscr{L}_B(X)$, we define an element $\check{t} \in (X \hat{\otimes} \overline{X} \hat{\otimes} B_*)^*$ (here, $\overline{X}$ is the conjugate space of $X$ and $\hat{\otimes}$ denotes the projective tensor product of Banach spaces) by
$$\check{t}(x\otimes \overline{y}\otimes \chi)= \chi(\langle y,tx\rangle), \quad x,y\in X, \quad \chi \in B_*.$$
Then the map
$\mathscr{L}_B(X)\to (X\hat{\otimes}\overline{X}\hat{\otimes} B_*)^*: t \mapsto \check{t}$
is an isometry with weak$^*$-closed image, and thus $\mathscr{L}_B(X)$ has a predual.

In \cite{Pa73}*{Theorem 3.2}, it is proven that if $B$ is a von Neumann algebra and $X$ is a right pre-Hilbert $B$-module, then there exists a canonical self-dual $B$-module $\overline{X}$ containing $X$ (in an inner-product preserving way). We will refer to $\overline{X}$ as the \emph{self-dual completion} of $X$. We will require the following two lemmas involving self-dual completions:

\begin{Lem}\label{self-duality}
    Let $B_1, B_2$ be von Neumann algebras and $\phi: B_1\to B_2$ a $*$-isomorphism. Suppose that $X_1$ is a right pre-Hilbert $B_1$-module and $X_2$ a right pre-Hilbert $B_2$-module. If $\kappa: X_1\to X_2$ is a bounded linear map satisfying 
    $$\kappa(xb) = \kappa(x)\phi(b), \quad x \in X_1, \quad b \in B_1,$$
    then there exists a unique bounded linear map
    $\overline{\kappa}: \overline{X_1}\to \overline{X_2}$
    that extends $\kappa$. It satisfies
    $$\|\overline{\kappa}\|= \|\kappa\|, \quad \overline{\kappa}(x b)= \overline{\kappa}(x)\phi(b), \quad x \in \overline{X_1}, \quad b\in B_1.$$
    If moreover $\kappa$ is an isometric surjection, then so is $\overline{\kappa}$.
\end{Lem}
\begin{proof}
    We can mimic the proof of \cite{Pa73}*{Proposition 3.6}, making use of the explicit construction of the self-dual completion. We only comment on the fact that if $\kappa$ is an isometric surjection, then so is $\overline{\kappa}$, as this is not explicitly mentioned in \cite{Pa73}. Indeed, in that case, uniqueness of the extension implies that $\overline{\kappa}$ is invertible with $\overline{\kappa}^{-1}= \overline{\kappa^{-1}}$. Consequently, if $x\in \overline{X_1}$, then using that $\|\overline{\kappa}\| = 1 = \|\overline{\kappa}^{-1}\|$, we find
    $$\|x\| = \|\overline{\kappa}^{-1}(\overline{\kappa} x)\| \le \|\overline{\kappa}(x)\|\le \|x\|,$$
    whence $\overline{\kappa}$ is isometric.
\end{proof}

\begin{Lem}\label{s-topology}
    Let $B$ be a von Neumann algebra, $X$ a pre-Hilbert $B$-module with self-dual completion $\overline{X}$. Given $\chi \in B_*$ and $x,y \in X$, we write
$$\omega_{\chi, x,y}: \mathscr{L}_{B}(\overline{X})\to \mathbb{C}: t \mapsto \chi(\langle x, ty\rangle).$$
Then $\operatorname{span}\{\omega_{\chi, x, y}: \chi\in B_*, x,y \in X\}$ is norm-dense in $\mathscr{L}_B(\overline{X})_*$.
\end{Lem}
\begin{proof} It is clear 
that $\omega_{\xi, x,y}\in \mathscr{L}_B(\overline{X})_*$ for all $\chi\in B_*$ and all $x,y \in X$. Therefore, 
$$\overline{\operatorname{span}}\{\omega_{\chi, x,y}: \chi\in B_*, x,y \in X\}\subseteq \mathscr{L}_B(\overline{X})_*.$$
If equality does not hold, then by the Hahn-Banach theorem and the fact that $(\mathscr{L}_B(\overline{X})_*)^*\cong \mathscr{L}_B(\overline{X})$, there exists a non-zero $t\in \mathscr{L}_B(\overline{X})$ such that
   $$\chi(\langle x,ty\rangle)= 0, \quad \chi \in B_*, \quad x,y\in X.$$
   But then clearly $t\vert_X = 0$. By the argument in the uniqueness statement in the proof of \cite{Pa73}*{Proposition 3.6}, this implies that $t= 0$, which is a contradiction.
\end{proof}

Let $(A, \alpha)$ be a $\G$-$W^*$-algebra. Recall the notation $A^\alpha:= \{a\in A: \alpha(a)= a \otimes 1\}$ for the fixed point von Neumann subalgebra. Note that the space $\mathcal{R}(A)$ becomes a right pre-Hilbert $A^\alpha$-module for the action by multiplication and the $A^\alpha$-valued inner product
$$\langle x, y\rangle := (\id \otimes \Phi)\alpha(x^*y), \quad x,y\in \mathcal{R}(A).$$
We have a natural injective $*$-homomorphism
$$\mathcal{R}(A)\to \mathscr{L}_{A^\alpha}(\mathcal{R}(A)): a \mapsto (x \mapsto ax).$$
Composing with the natural injective $*$-homomorphism
$\mathscr{L}_{A^\alpha}(\mathcal{R}(A))\hookrightarrow \mathscr{L}_{A^\alpha}(\overline{\mathcal{R}(A)})$ \cite{Pa73}*{Corollary 3.7}, we obtain
the canonical isometric $*$-homomorphism
$$L_A: \mathcal{R}(A)\to \mathscr{L}_{A^\alpha}(\overline{\mathcal{R}(A)}).$$

\begin{Lem}
    The map
    $L_A: \mathcal{R}(A)\to \mathscr{L}_{A^\alpha}(\overline{\mathcal{R}(A)})$
    extends uniquely to an isometric normal unital $*$-homomorphism
    $L_A: A\to \mathscr{L}_{A^\alpha}(\overline{\mathcal{R}(A)}).$
\end{Lem}
\begin{proof}
     Consider the restriction map
    $\operatorname{Res}: A_*\to \mathcal{R}(A)^*: \omega \mapsto \omega\vert_{\mathcal{R}(A)}$
    which is isometric. We show that 
    \begin{equation}\label{step1}
        L_A^*(\mathscr{L}_{A^\alpha}(\overline{\mathcal{R}(A)})_*)\subseteq \operatorname{Res}(A_*).
    \end{equation} If $\chi\in (A^\alpha)_*$ and $p,q,r\in \mathcal{R}(A)$, we calculate that 
    \begin{align*}
        L_A^*(\omega_{\chi, p,q})(r)&= \chi(\langle p, rq))= \chi((\id \otimes \Phi)\alpha(p^* rq)),
    \end{align*} 
    from which it is clear that $L_A^*(\omega_{\chi, p,q})\in \operatorname{Res}(A_*)$. By Lemma \ref{s-topology}, we conclude that \eqref{step1} holds.
    
  Next, we show that $L_A^*(\mathscr{L}_{A^\alpha}(\overline{\mathcal{R}(A)})_*)$ is norm-dense in $\operatorname{Res}(A_*).$ Clearly, it suffices to show that the normal functionals
    \begin{align*}\label{normals}
        A \ni r \mapsto \chi((\id \otimes \Phi)\alpha(prq)), \quad p,q \in \mathcal{R}(A), \quad \chi\in (A^\alpha)_*
    \end{align*}
    span a norm-dense subspace of $A_*$. Assume to the contrary that this is not the case. By the Hahn-Banach theorem, there is a non-zero $r \in A$ such that
    $$\chi((\id \otimes \Phi)\alpha(prq)) = 0, \quad \forall \chi \in (A^\alpha)_*, \quad \forall p,q\in \mathcal{R}(A).$$
Hence, we deduce that 
$$(\id \otimes \Phi)\alpha(prq) = 0 \quad \forall p,q \in A.$$
Taking $p = r^*, q = 1$ and using faithfulness of $\Phi$, we deduce that $r= 0$, which is a contradiction. Therefore, the denseness is proven.

Form the composition
    $\theta:= \operatorname{Res}^{-1}\circ L_A^*: \mathscr{L}_{A^\alpha}(\overline{\mathcal{R}(A)})_*\to A_*.$
   Since $\theta$ has dense image, the adjoint $\theta^*: A\to \mathscr{L}_{A^\alpha}(\overline{\mathcal{R}(A)})$ is an injective normal $*$-homomorphism extending $L_A: \mathcal{R}(A)\to \mathscr{L}_{A^\alpha}(\overline{\mathcal{R}(A)}).$
\end{proof}

We can now prove the announced result:

\begin{proof}[Proof of Theorem \ref{important inbetween}.]  Let $\phi: \mathcal{R}(A)\to \mathcal{R}(B)$ be a $\G$-equivariant $*$-isomorphism. The fact that $\phi$ is $\G$-equivariant entails that
    $$\langle \phi(x), \phi(y)\rangle = \phi(\langle x,y\rangle), \quad x,y \in \mathcal{R}(A).$$ 
    Consequently, $\phi$ is an isometric surjection (with respect to the norms induced by the inner products). By Lemma \ref{self-duality}, there is a unique bounded linear map
    $\overline{\phi}: \overline{\mathcal{R}(A)}\to \overline{\mathcal{R}(B)}$ extending $\phi$. Moreover, $\overline{\phi}$ is an isometric surjection and satisfies 
    $$\overline{\phi}(xa) = \overline{\phi}(x) \phi(a), \quad x \in \overline{\mathcal{R}(A)}, \quad a \in A^\alpha.$$
    This implies (see e.g.\ \cite{La95}*{Chapter 3}) that 
    $$\langle \overline{\phi}(x), \overline{\phi}(y)\rangle = \phi(\langle x,y\rangle), \quad x,y \in \overline{\mathcal{R}(A)}.$$
    Therefore
    we have an induced $*$-isomorphism
    $$\Pi: \mathscr{L}_{A^\alpha}(\overline{\mathcal{R}(A)})\to \mathscr{L}_{B^\beta}(\overline{\mathcal{R}(B)}): t \mapsto \overline{\phi}\circ t\circ \overline{\phi}^{-1}.$$
    Then if $a\in \mathcal{R}(A)$, we have $\Pi(L_A(a))= L_B(\phi(a))$. Since $\Pi$ is a $\sigma$-weakly continuous homeomorphism, we find that 
    $$\Pi(L_A(A))= \Pi(\overline{L_A(\mathcal{R}(A))}^{\sigma\text{-weak}})= \overline{L_B(\mathcal{R}(B))}^{\sigma\text{-weak}} = L_B(B).$$
   Hence, the $*$-isomorphism
   $L_B^{-1}\circ \Pi\circ L_A: A\to B$ is the desired extension of $\phi$.
\end{proof}

We are now in a position to prove the main result of this section. In the proof, we will use the following elementary observation:

\begin{Lem}\label{cornering}
    Let $(A, \alpha)$ be a $\G$-$W^*$-algebra and $p\in A^\alpha$ a projection. Then 
    $\mathcal{R}(pAp) = p\mathcal{R}(A)p.$
\end{Lem}
\begin{proof}
    Given an irreducible representation $\pi\in \Irr(\G)$ and a $\G$-$W^*$-algebra $(B, \beta)$, we write $B_\pi$ for the associated spectral subspace. It is then easily verified that $(pAp)_\pi = pA_\pi p$ from which the result easily follows.
\end{proof}

We need some further notations:
\begin{itemize}
    \item If $U_\pi\in B(\mathcal{H}_\pi)\ovot L^\infty(\G)$ is a unitary representation of $\G$ and $(B, \beta)$ is a $\G$-$W^*$-algebra, then we obtain a $\G$-action 
$$\gamma_\pi: B(\mathcal{H}_\pi)\ovot B \to B(\mathcal{H}_\pi)\ovot B \ovot  L^\infty(\G): z \mapsto U_{\pi, 13}(\id \otimes \beta)(z)U_{\pi, 13}^*.$$
Given $\xi, \eta \in \mathcal{H}_\pi$, we write $U_\pi(\xi, \eta):= (\omega_{\xi, \eta} \otimes \id)(U_\pi)\in L^\infty(\G)$.
\item Given a $\G$-$C^*$-algebra $(X, \alpha)$ and $n\ge 1$, consider 
$$\alpha^{(n)}: M_n(X)\to M_n(X\otimes C(\G))\cong M_n(X)\otimes C(\G): [x_{ij}]\mapsto [\alpha(x_{ij})].$$
Then it is easily verified that $(M_n(X), \alpha^{(n)})$ is also a $\G$-$C^*$-algebra. The same applies for $\G$-$W^*$-algebras.
\end{itemize}

\begin{Theorem}\label{ergodic} Let $\G$ be a compact quantum group and let $(A, \alpha)$ and $(B, \beta)$ be two $\G$-$W^*$-algebras. Consider the following statements:
    \begin{enumerate}\setlength\itemsep{-0.5em}            \item $(\mathcal{R}(A), \alpha)$ and $(\mathcal{R}(B), \beta)$ are $\G$-$C^*$-Morita equivalent.
        \item $(A, \alpha)$ and $(B, \beta)$ are $\G$-$W^*$-Morita equivalent.
        \item There exists an irreducible $\G$-representation $U_\pi\in B(\mathcal{H}_\pi)\odot \mathcal{O}(\G)$, a projection $p \in (B(\mathcal{H}_\pi)\odot B)^{\gamma_\pi}$ and a $\G$-equivariant $*$-isomorphism $(A, \alpha)\cong (p(B(\mathcal{H}_\pi)\odot B)p, \gamma_\pi).$
    \end{enumerate}
\end{Theorem}
Then $(1)\implies (2)$. If $(A, \alpha)$ and $(B, \beta)$ are both ergodic, then $(1)\iff (2) \iff (3)$.
\begin{proof} $(1)\implies (2)$ Assume that $(\mathcal{R}(A), \alpha)$ and $(\mathcal{R}(B), \beta)$ are $\G$-$C^*$-Morita equivalent. To simplify notation, let us write $X:= \mathcal{R}(A)$ and $Y:= \mathcal{R}(B)$.
Since $X$ and $Y$ are $\G$-$C^*$-Morita equivalent, a standard argument involving unitality (see \cite{WO93}*{15.4.3 Remarks}) shows that there is an equivalence $\G$-$X$-$Y$-bimodule of the form $pY^{\oplus n}$ where $n \ge 1$ and $p\in M_n(Y)$ is a norm-full projection. Composing with the canonical equivalence $\G$-$Y$-$M_n(Y)$-module $Y^{\oplus n}$, we see that $p M_n(Y)$ with its obvious right $M_n(Y)$-module structure can be upgraded to an equivalence $\G$-$X$-$M_n(Y)$-bimodule. Arguing as in the proof of Theorem \ref{char2} (using the linking $\G$-$C^*$-algebra $\mathcal{K}_{M_n(Y)}(pM_n(Y)\oplus M_n(Y))$ associated to the equivalence $\G$-$X$-$M_n(Y)$-bimodule $p M_n(Y)$), it follows that there exists $P \in M_n(Y)\otimes C(\G)$ such that
$$(\id \otimes \Delta)(P) = P_{12} (\beta^{(n)}\otimes \id)(P), \quad P^*P= \beta^{(n)}(p), \quad PP^* = p\otimes 1$$
and a $*$-isomorphism
$\phi: X \to p M_n(Y) p$
such that 
$$(\phi\otimes \id)\alpha(x)= P \beta^{(n)}(\phi(x))P^*, \quad x \in X.$$
Consider the $\G$-$W^*$-algebra $C:= pM_n(B)p$ with its natural $\G$-action determined by
$$\gamma(c)= P \beta^{(n)}(c) P^*, \quad c \in C.$$
Writing $Z:= \mathcal{R}(C)$, we now prove that $Z= p M_n(Y) p$. Indeed, since $(p M_n(Y)p, \gamma)$ is a $\G$-$C^*$-algebra, it immediately follows that $p M_n(Y)p \subseteq Z$. Conversely, if $c\in C$ and $\omega \in L^1(\G)$, then from the fact that $P = (p\otimes 1)P \in M_n(Y)\otimes C(\G)$, we find that
\begin{align*}
    (\id \otimes \omega)\gamma(c)&= p(\id \otimes \omega)(P \beta^{(n)}(c) P^*)p\\
    &\in p[M_n(Y)\underbrace{(\id \otimes \omega)((1\otimes C(\G))\beta^{(n)}(c)(1\otimes C(\G)))}_{\subseteq \mathcal{R}(M_n(B)) = M_n(Y)} M_n(Y)]p\subseteq p M_n(Y)p.
\end{align*}

Hence, $\phi$ defines a $\G$-equivariant $*$-isomorphism
$(\mathcal{R}(A), \alpha)\cong (\mathcal{R}(C), \gamma).$
By Theorem \ref{important inbetween}, it extends to a $\G$-equivariant isomorphism $(A, \alpha)\cong (C, \gamma)$. Consequently, Theorem \ref{char2} implies that $(A, \alpha)$ and $ (B, \beta)$ are $\G$-$W^*$-Morita equivalent.

For the remainder of the proof, we assume that both $(A, \alpha)$ and $(B, \beta)$ are ergodic. We adapt the argument in \cite{DC12}*{Proposition 1.4} to the von Neumann algebra setting. 

$(2)\implies (3)$ Suppose that $\mathcal{H}\in \Corr^\G(A,B)$ is a $\G$-$W^*$-Morita correspondence. Put
    $$X:= \mathscr{L}_B(L^2(B), \mathcal{H})=\{x\in B(L^2(B),\mathcal{H})\mid \forall b\in B: x\rho_B(b)= \rho_{\mathcal{H}}(b)x\},$$ which we endow with its natural coaction
    $$\delta: X \to X \ovot L^\infty(\G): x \mapsto U_\mathcal{H}(x\otimes 1)U_\beta^*.$$
Considering the decomposition of $X$ into spectral subspaces, there is an irreducible $\G$-representation $\pi$ and $x_1, \dots, x_{n_\pi}\in X$, not all zero, such that 
$$\delta(x_i) = \sum_{j=1}^{n_\pi} x_j\otimes U_\pi(e_j^\pi, e_i^\pi),\quad i=1, \dots, n_\pi,$$
 where $\{e_i^\pi\}_{i=1}^{n_\pi}$ is an orthonormal basis for $\mathcal{H}_\pi$. We consider the map
  $$t: X \to \mathcal{H}_\pi\odot B: x \mapsto \sum_{i=1}^{n_\pi} e_i^\pi\otimes \pi_B^{-1}(x_i^* x).$$
  Then $t \in \mathscr{L}_B(X, \mathcal{H}_\pi\odot B)$, with adjoint given by
  $$t^*: \mathcal{H}_\pi\odot B \to X: \eta \otimes b \mapsto \sum_{i=1}^{n_\pi} \langle e_i^\pi, \eta\rangle x_i \pi_B(b).$$
  Writing $\lambda:= \sum_{i=1}^{n_\pi}\langle x_i, x_i\rangle_A \in A^{\alpha}= \C$, we see that $t^*t = \lambda \id_X$. Replacing $x_i$ by $x_i/\sqrt{\lambda}$, we may assume that the elements $x_i$ are chosen so that $t$ is isometric. A simple calculation shows that 
  $$p= tt^*= \sum_{k,l=1}^{n_\pi} E_{kl}^\pi\otimes \pi_B^{-1}(x_k^* x_l) \in (B(\mathcal{H}_\pi)\odot B)^{\gamma_\pi},$$
  where $\{E_{kl}^\pi\}$ are the matrix units with respect to the basis $\{e_i^\pi\}_{i=1}^{n_\pi}.$ By \cite{Rie74}*{Proposition 7.6}, the map
  $$A \to \mathscr{L}_B(X): a \mapsto (x \mapsto \pi_\mathcal{H}(a)x)$$
  is a $*$-isomorphism. Therefore, we obtain the map
$$\phi: A\cong \mathscr{L}_B(X)\cong \mathscr{L}_B(p(\mathcal{H}_\pi\odot B)) \cong p \mathscr{L}_B(\mathcal{H}_\pi\odot B) p \cong p(B(\mathcal{H}_\pi)\odot B) p.$$ 
A routine calculation shows that the map $$\phi: (A, \alpha)\to (B(\mathcal{H}_\pi)\odot B, \gamma_\pi): a \mapsto \sum_{k,l=1}^{n_\pi} E_{kl}^\pi\otimes \pi_B^{-1}(x_k^* \pi_{\mathcal{H}}(a) x_l)$$
is $\G$-equivariant. Thus, $\phi$ defines a $\G$-equivariant isomorphism $(A, \alpha)\cong_\G (p(B(\mathcal{H}_\pi)\odot B)p, \gamma_\pi)$.

$(3)\implies (1)$ The isomorphism $A\cong_\G p(B(\mathcal{H}_\pi)\odot B)p$ restricts to a $*$-isomorphism
$$\mathcal{R}(A)\cong_\G \mathcal{R}(p(B(\mathcal{H}_\pi)\odot B)p)= p \mathcal{R}(B(\mathcal{H}_\pi)\odot B)p = p(B(\mathcal{H}_\pi)\odot \mathcal{R}(B))p,$$
where we made use of Lemma \ref{cornering}. But then \cite{DC12}*{Proposition 1.4} implies that $(\mathcal{R}(A), \alpha)$ and $(\mathcal{R}(B), \beta)$ are $\G$-$C^*$-Morita equivalent.\end{proof}

Let us recall that if $(X, \alpha)$ is an ergodic $\G$-$C^*$-algebra and $\phi: (X,\alpha) \to (\C, \tau)$ is the unique $\G$-equivariant state, then the associated GNS-representation
$\pi_\phi: X  \to B(L^2(X, \phi))$
is faithful, and the von Neumann algebra $B:= \pi_\phi(X)''$ becomes a $\G$-$W^*$-algebra for the $\G$-action uniquely determined by
$$\beta: B \to B \ovot L^\infty(\G), \quad \beta(\pi_\phi(x))= (\pi_\phi\otimes \id)\alpha(x), \quad x \in X.$$
One can show that $\mathcal{R}(B) = \pi_\phi(X)\cong X$. Keeping this in mind, the following result is then an immediate consequence of Theorem \ref{ergodic}:

\begin{Cor}\label{classification} The classification of Podleś spheres up to $SU_q(2)$-$C^*$-Morita equivalence, as established in \cite{DC12}*{Theorem 0.1}, is also valid on the von Neumann algebra level. 
\end{Cor}

\textbf{Acknowledgments:} The research of the author was supported by Fonds voor Wetenschappelijk Onderzoek (Flanders), via an FWO Aspirant fellowship, grant 1162524N. The author would like to thank Kenny De Commer for valuable suggestions and discussions
throughout this entire project, as well as Benjamin Anderson-Sackaney and Jacek Krajczok for a useful discussion. Finally, the author would like to thank the anonymous referee for providing useful suggestions.


\begin{thebibliography}{00}
\bibitem[AK24]{AK24} B. Anderson-Sackaney and F. Khosravi, Topological Boundaries of Representations and Coideals, \emph{Adv. Math.} 
\textbf{425} (2024), 109830.
\bibitem[AV23]{AV23} B. Anderson-Sackaney and L. Vainerman, Fusion Modules and Amenability of Coideals of Compact and Discrete Quantum Groups, \emph{preprint}, arXiv:2308.01656.
\bibitem[BDH88]{BDH88} M. Baillet, Y. Denizeau and J.F. Havet, Indice d'une esperance conditionnelle, \emph{Compositio Math} \textbf{66} (1988), 199--236.
\bibitem[BKKO17]{BKKO17} E. Breuillard, M. Kalantar, M. Kennedy and N. Ozawa, C*-simplicity and the unique trace property for discrete groups, \emph{Publ. Math. Inst. Hautes \'Etudes Sci.} \textbf{126} (2017), 35--71.
%\bibitem[Bl97]{Bl97} D.P. Blecher, On self-dual Hilbert modules, Operator algebras and their applications ({W}aterloo, {ON},
              1994/1995), \emph{Fields Inst. Commun.} \textbf{13} (1997), Amer. Math. Soc., Providence, RI, 65--80
\bibitem[BS93]{BS93} S. Baaj and G. Skandalis, Unitaires multiplicatifs et dualit\'{e} pour les produits crois\'{e}s de C$^*$-alg\`{e}bres, \emph{Ann. scient. Ec. Norm. Sup.}, 4e s\'{e}rie \textbf{26} (1993), 425--488.
\bibitem[Bro03]{Bro03} R.M. Brouwer, A bicategorical approach to Morita equivalence for von Neumann algebras, \emph{J. Math. Phys.} \textbf{44} (2003), 2206--2214.
\bibitem[Con80+]{Con80+} A. Connes, Correspondences, \emph{unpublished}.
\bibitem[Cr17a]{Cr17a} J. Crann, Amenability and covariant injectivity of locally compact quantum groups {II}, \emph{Canad. J. Math.} \textbf{69} (2017), 1064--1086.
\bibitem[Cr17b]{Cr17b} J. Crann, On hereditary properties of quantum group amenability, \emph{Proc. Amer. Math. Soc.} \textbf{145} (2017), 627--635.
\bibitem[CN16]{CN16} J. Crann and M. Neufang, Amenability and covariant injectivity of locally compact quantum groups, Trans. Amer. Math. Soc. \textbf{368} (2016), 495--513.
\bibitem[DC09]{DC09} K. De Commer, Galois coactions for algebraic and locally compact quantum groups, PhD thesis (2009), \url{https://lirias.kuleuven.be/retrieve/75632}.
\bibitem[DC12]{DC12} K. De Commer, Equivariant Morita equivalences between Podleś spheres, \emph{Banach Center Publ}, \textbf{98} (2012), 85--105.
\bibitem[DC17]{DC17} K. De Commer, Actions of compact quantum groups, Topological quantum groups, \textit{Banach Center Publ.} \textbf{111} (2017), Polish Acad. Sci. Inst. Math., Warsaw, 33--100.
\bibitem[DCDR24]{DCDR24a} K. De Commer and J. De Ro, Approximation properties for dynamical W$^*$-correspondences, \emph{Adv. Math.} \textbf{458} (2024), 109958.
\bibitem[DCDR25]{DCDR24b} K. De Commer and J. De Ro, Amenable actions of compact and discrete quantum groups on von Neumann algebras, \emph{J. Funct. Anal.} \textbf{289} (2025), Paper No. 110973, 39.
\bibitem[DH24]{DH24} J. De Ro and L. Hataishi, Actions of compact and discrete quantum groups on operator systems, \emph{Int. Math. Res. Not.} \textbf{15} (2024), 11190–11220.
\bibitem[DKSS12]{DKSS12} M. Daws, P. Kasprzak, A. Skalski and P. M. Sołtan, Closed quantum subgroups of locally compact quantum groups, \emph{Adv. Math.} \textbf{231} (2012), 3473--3501.
\bibitem[DR24a]{DR24a} J. De Ro, Equivariant injectivity of crossed products, \emph{J. Oper. Theory} (in press).
\bibitem[DR24b]{DR24b} J. De Ro, A categorical interpretation of Morita equivalence for dynamical von Neumann algebras, \emph{J. Algebra} \textbf{666} (2025), 673-702.
\bibitem[GLR85]{GLR85} P. Ghez, R. Lima and J. Roberts, $W^*$-categories, \emph{Pac. J. Math.}, \textbf{120(1)} (1985), 79--109.
\bibitem[ILS98]{ILS98} M. Izumi, Masaki and R. Longo and S. Popa, A {G}alois correspondence for compact groups of automorphisms of von {N}eumann algebras with a generalization to {K}ac algebras, \emph{J. Funct. Anal.} \textbf{155} (1998), 25--63.
 \bibitem[KKSV22]{KKSV22} M. Kalantar, P. Kasprzak, A. Skalski, and R. Vergnioux, Noncommutative {F}urstenberg boundary, \emph{Anal. PDE} \textbf{15} (2022), 795--842. 
 \bibitem[KS14]{KS14} P. Kasprzak and P.M. Sołtan, Embeddable quantum homogeneous spaces, \emph{J. Math. Anal. Appl.} \textbf{411} (2014), 574--591.
\bibitem[KS15]{KS15} P. Kasprzak and P.M. Sołtan, Quantum groups with projection on von Neumann algebra level, \emph{J. Math. Anal. Appl.} \textbf{427} (2015), 289--306.
\bibitem[KV00]{KV00} J. Kustermans and S. Vaes, Locally compact quantum groups, \emph{Ann. Sci. \'{E}c. Norm. Sup\'{e}r.} \textbf{33} (6) (2000), 837--934.
\bibitem[KV03]{KV03} J. Kustermans and S. Vaes, Locally compact quantum groups in the von Neumann algebraic setting, \emph{Math. Scand.} \textbf{92} (1) (2003), 68--92.
\bibitem[La95]{La95} E.C. Lance, Hilbert C*-modules, \emph{London Mathematical Society Lecture Note Series}, Cambridge University Press, Cambridge, 1995.
\bibitem[Moa18]{Moa18} M. S. M. Moakhar, Amenable actions of discrete quantum groups on von Neumann algebras, \emph{preprint}, arXiv:1803.04828.
\bibitem[NV10]{NV10} R. Nest and C. Voigt, Equivariant {P}oincar\'e{} duality for quantum group actions, \emph{J. Funct. Anal.} \textbf{258} (2010), 1466--1503.
\bibitem[Pa73]{Pa73} W.L. Paschke, Inner product modules over B*-algebras, \emph{Trans. Amer. Math. Soc.} \textbf{182} (1973), 443--468.
%\bibitem[PS17]{PS17} S. Pekka and A. Skalski, Actions of locally compact (quantum) groups on ternary rings of operators, their crossed products, and generalized Poisson boundaries, \emph{Kyoto J. Math.} \textbf{57} (2017), 667--691.
\bibitem[Rie74]{Rie74} M.A. Rieffel, Morita equivalence for C$^*$-algebras and W$^*$-algebras, \emph{J. Pure Appl. Alg.} \textbf{5} (1974), 51--96.
\bibitem[Rua96]{Rua96} Z.-J. Ruan, Amenability of Hopf von Neumann algebras and Kac algebras, \emph{J. Funct. Anal.} \textbf{139} (1996), 466--499.
\bibitem[Tak03]{Tak03} M. Takesaki, Theory of Operator Algebras II, \emph{Springer, Berlin} (2003).
\bibitem[Va01]{Va01} S. Vaes, The unitary implementation of a locally compact quantum group action, \emph{J. Funct. Anal.} \textbf{180} (2001), 426--480.
\bibitem[Vae05]{Vae05} S. Vaes, A new approach to induction and imprimitivity results, \emph{J. Funct. Anal.} \textbf{229} (2005), 317--374.
\bibitem[VV03]{VV03} S. Vaes and A. Van Daele, The Heisenberg commutation relations, commuting squares and the Haar measure on locally compact quantum groups, In: \emph{Operator algebras and mathematical physics: conference proceedings, Constanta (Romania), July 2-7, 2001}, Editors J.-M. Combes, J. Cuntz, G.A. Elliott, G. Nenciu, H. Siedentop and S. Stratila, \emph{Theta Foundation}, Bucarest (2003), 379--400.
\bibitem[WO93]{WO93} N.E. Wegge-Olsen, {$K$}-theory and {$C^*$}-algebras, \emph{The Clarendon Press, Oxford University Press, New York} (1993).
\end{thebibliography}
\end{document}